\documentclass[a4paper,11pt]{amsart}
\usepackage{a4wide}

\usepackage{amssymb}
\usepackage{amsmath}
\usepackage{amsthm}

\allowdisplaybreaks

\usepackage{amsfonts}
\usepackage{amssymb}
\usepackage{amsmath}
\usepackage{amsthm}
\usepackage{bbm}
\usepackage{verbatim}
\usepackage{url}
\usepackage[latin1]{inputenc} 
\usepackage[T1]{fontenc}

\usepackage{amsfonts}
\usepackage{amsmath,amssymb,amsthm}

\usepackage{mathrsfs}
\usepackage{verbatim}

\usepackage{graphicx}
\usepackage{ifpdf}

\DeclareMathOperator{\ch}{Ch}

\newtheorem{thm}{Theorem}[section]
\newtheorem{prop}[thm]{Proposition}
\newtheorem{lem}[thm]{Lemma}
\newtheorem{cor}[thm]{Corollary}
\newtheorem{lemma}[thm]{Lemma}

\theoremstyle{definition} 

\newtheorem{defn}[thm]{Definition}

\newcommand{\pa}{\varphi^t}

\newcommand{\R}{\mathbb{R}}

\newcommand{\D}{\mathscr{D}}

\newcommand{\XX}{X}

\newcommand{\bmo}{{\rm BMO}}
\newcommand{\lbmo}{{\rm bmo}}

\def\cg{{\mathcal G}}

\def\nhoq{{h^{1,\,q}(\rlz)}}

\newcommand{\GGp}{{\mathop \cg\limits^{\circ}}_{1,\,1}}

\newcommand{\X}{X_1\times X_2}

\numberwithin{equation}{section}

\def\rr{{\mathbb R}}

\def\zz{{\mathbb Z}}
\def\nn{{\mathbb N}}

\def\ch{{\mathcal H}}

\def\cm{{\mathcal M}}

\def\fz{\infty}
\def\az{\alpha}
\def\supp{{\mathop\mathrm{\,supp\,}}}

\def\loc{{\mathop\mathrm{\,loc\,}}}

\def\lz{\lambda}

\def\ez{\epsilon}

\def\vz{\varphi}

\def\pa{\partial}
\def\wz{\widetilde}

\def\rizo{{R_{\Delta_\lz,\,1}}}
\def\rizt{{R_{\Delta_\lz,\,2}}}

\def\lozd{{L^1(\rlz)}}

\def\lpzd{{L^p(\rlz)}}

\def\ls{\lesssim}
\def\gs{\gtrsim}

\def\tbz{{\triangle_\lz}}
\def\dmz{{dm_\lz}}

\def\riz{{R_{\Delta_\lz}}}

\def\wrizo{\wz {R_{\Delta_\lz,1 }}}
\def\wrizt{\wz {R_{\Delta_\lz,2 } }}

\def\lpz{{L^p(\rr_+,\, dm_\lz)}}

\def\dsum{\displaystyle\sum}
\def\dint{\displaystyle\int}

\def\dfrac{\displaystyle\frac}

\def\r{\right}
\def\lf{\left}

\def\noz{\nonumber}

\def\rlz{\R_\lz}

\def\ltp{{L^2\,(\rlz)}}

\def\hop{{h^1\,(\rlz)}}

\def\XXint#1#2#3{{\setbox0=\hbox{$#1{#2#3}{\int}$}
     \vcenter{\hbox{$#2#3$}}\kern-.5\wd0}}

\def\mbhh{{[ b, \rizo\rizt]}}

\def\hop{{h^1\,(\rlz)}}
\def\bmop{{{\rm bmo}\,(\rlz)}}

\allowdisplaybreaks
\begin{document}

\title[Product BMO, little BMO and Riesz commutators]{Product BMO, little BMO and Riesz commutators \\
in the Bessel setting}

\author{Xuan Thinh Duong}
\address{Xuan Thinh Duong, Department of Mathematics\\
         Macquarie University\\
         NSW 2019\\
         Australia
         }
\email{xuan.duong@mq.edu.au}

\author{Ji Li}
\address{Ji Li, Department of Mathematics\\
         Macquarie University\\
         NSW 2019\\
         Australia
         }
\email{ji.li@mq.edu.au}

\author{Yumeng Ou}
\address{Yumeng Ou, Department of Mathematics\\
         Massachusetts Institute of Technology\\
         Cambridge, MA 02139 USA
         }
\email{yumengou@mit.edu}

\author{Brett D. Wick}
\address{Brett D. Wick, Department of Mathematics\\
         Washington University -- St. Louis\\
         St. Louis, MO 63130-4899 USA
         }
\email{wick@math.wustl.edu}

\author{Dongyong Yang$^\ast$}
\address{Dongyong Yang(Corresponding author), School of Mathematical Sciences\\
 Xiamen University\\
  Xiamen 361005,  China
  }
\email{dyyang@xmu.edu.cn }




\date{\today}



\thanks{$\ast$ Corresponding author}

\begin{abstract}
In this paper, we study the product BMO space, little bmo space and their connections with the corresponding commutators associated with Bessel operators studied by Weinstein, Huber, and by Muckenhoupt--Stein.  We first prove that the product BMO space in the Bessel setting can be used to deduce the boundedness of the iterated commutators with the Bessel Riesz transforms.  We next study the little $\rm bmo$ space in this Bessel setting and obtain the equivalent characterization of this space in terms of commutators, where the main tool that we develop is the characterization of the predual of little bmo and its weak factorizations.  We further show that in analogy with the classical setting, the little $\rm bmo$ space is a proper subspace of the product $\rm BMO$ space. These extend the previous related results studied by Cotlar--Sadosky and Ferguson--Sadosky on the bidisc to the Bessel setting, where the usual analyticity and Fourier transform do not apply.
\end{abstract}

\maketitle




\section{Introduction}
\label{sec:introduction}
\setcounter{equation}{0}

The study of commutators of multiplication operators with Calder\'on--Zygmund operators has its roots in complex function theory and Hankel operators.  This was later extended to the case of general Calder\'on--Zygmund operators by Coifman, Rochberg and Weiss \cite{crw}, who showed that the space of bounded mean oscillation introduced by John and Nirenberg is characterized by a family of commutators:
$$
\left\Vert b\right\Vert_{\bmo\left(\mathbb{R}^n\right)}\approx \max_{1\leq j\leq n} \left\Vert [M_b,R_j]\right\Vert_{L^p(\mathbb{R}^n)\to L^p(\mathbb{R}^n)}
$$
where $R_j$ is the $j$th Riesz transform.  Results of this type have then been extended by Uchiyama to handle spaces of homogeneous type under certain assumptions on the measures and to show that a single Hilbert transform (Riesz transform) actually characterizes $\rm BMO$ \cite{U}.  These results were further extended to the multiparameter setting showing that the product $\rm BMO$ space of Chang and Fefferman can also be characterized by iterated commutators (see Hilbert transform in \cite{fl} and Riesz transforms in \cite{LPPW}) and little $\rm bmo$ by the boundedness of two commutators (see Hilbert transform in \cite{fs} and Riesz transforms in \cite{DLWY3}).  The analysis here is intimately connected to the underlying space $\mathbb{R}^n$ and to the fact that the Riesz transforms are connected to a particular differential operator, the Laplacian.


In 1965, B. Muckenhoupt and E. Stein in \cite{ms} introduced harmonic analysis associated with Bessel operator $\tbz$, defined by
setting for suitable functions $f$,
\begin{equation*}
\tbz f(x):=\frac{d^2}{dx^2}f(x)+\frac{2\lz}{x}\frac{d}{dx}f(x),\quad \lz>0,\quad x\in \R_+:=(0,\fz).
\end{equation*}
The related elliptic partial differential equation is the following ``singular Laplace equation''
\begin{equation}\label{bessel laplace equation}
\triangle_{t,\,x} (u) :=\pa_{t}^2u + \pa_{x}^2u+\frac{2\lz}{x}\pa_{x}u=0
\end{equation}
studied by A. Weinstein \cite{w}, and A. Huber \cite{Hu} in higher dimensions, where they considered the generalised axially symmetric potentials, and obtained the properties of the solutions of this equation, such as the extension, the uniqueness theorem, and the boundary value problem for certain domains.
%
%
In \cite{ms} they developed a theory in the setting of
$\tbz$ which parallels the classical one associated to the standard Laplacian, where results on $\lpz$-boundedness of conjugate
functions and fractional integrals associated with $\tbz$ were
obtained for $p\in[1, \fz)$ and $\dmz(x):= x^{2\lz}\,dx$.

We also point out that Haimo \cite{h} studied the Hankel convolution transforms $\vz \sharp_\lz f$ associated with the Hankel transform in the Bessel setting systematically,
which provides a parallel theory to the classical convolution and Fourier transforms. It is well-known that
the Poisson integral of $f$ studied in \cite{ms} is the Hankel convolution of the Poisson kernel with $f$, see \cite{bdt}.
Since then, many problems in the Bessel context were studied, such as the boundedness of the Bessel Riesz transform,
Littlewood--Paley functions, Hardy and BMO spaces associated with Bessel operators, $A_p$ weights associated with Bessel operators
(see, for example, \cite{k78,ak,bfbmt,v08,bfs,bhnv,bcfr,yy,DLWY,DLWY2,DLMWY} and the references therein).

The aim of this paper is to study the product BMO and little bmo spaces via Riesz commutators in the Bessel setting.
In particular, the two main results we obtain can be seen as the analogs in the Bessel setting of the corresponding results in the classical setting.  Notably in our proof we bypass the use of analyticity and Fourier transform since they are not applicable in this Bessel operator setting.  We first show that the product $\rm BMO$ space in the Bessel setting can be used to prove the boundedness of the iterated commutators with the Bessel Riesz transforms.  We next study the little $\rm bmo$ space in this Bessel setting and obtain the equivalent characterization of this space in terms of commutators.  We further show, again in analogy with the classical setting, that the little $\rm bmo$ space is a proper subspace of the product $\rm BMO$ space.

To be more precise, for every interval $I\subset \R_+$, we denote it by $I:=I(x,t):= (x-t,x+t)\cap \R_+$. The measure of $I$ is defined as
$m_\lz(I(x,t)):=\int_{I(x,\,t)} y^{2\lz} dy$.  And recall that the Riesz transform $\riz(f)$ is defined as follows
\begin{align}\label{riz}
\riz(f)(x):= \int_{\R_+}  -\dfrac{2\lz}{\pi}\dint_0^\pi\dfrac{(x-y\cos\theta)(\sin\theta)^{2\lz-1}}
{(x^2+y^2-2xy\cos\theta)^{\lz+1}}\,d\theta \ f(y) \dmz(y).
\end{align}
In the product setting $\mathbb{R}_+\times \mathbb{R}_+$, we define $d\mu_\lz(x_1, x_2):=dm_\lz(x_1)\times dm_\lz(x_2)$ and
$ \rlz:= (\mathbb{R}_+\times \mathbb{R}_+,d\mu_\lz(x_1, x_2)).$ We denote by $R_{\Delta_\lambda,1}$ the Riesz transform on the first variable and $R_{\Delta_\lambda,2}$ the second.



The first main result of this paper is the upper bound of the iterated Riesz commutators $[ [b, R_{\Delta_\lambda,1}],R_{\Delta_\lambda,2}] $ in terms of product BMO space $\bmo_{\Delta_\lambda} (\rlz)$. For the definition of $\bmo_{\Delta_\lambda} (\rlz)$ we refer to Definition 2.5 in Section 2.

\begin{thm}\label{thm main1}
Let $b\in \bmo_{\Delta_\lambda} (\rlz)$. Then we have
\begin{align}
\|[ [b, R_{\Delta_\lambda,1}],R_{\Delta_\lambda,2}] \|_{L^2(\rlz)\to L^2(\rlz)} \leq C \|b\|_{\bmo_{\Delta_\lambda} (\rlz)}.
\end{align}
\end{thm}
For simplicity we only state the result for the case of two iterations; though the proof we provide works just as well for any number of parameters.

The proof strategy we employ to show this result is now the standard way to prove upper bounds for commutator estimates, see for example \cite{LPPW, LPPW2} and \cite{DO} in the Euclidean setting.  We express the Riesz transforms as averages of Haar shift type operators and then study the boundedness of the commutator with each Haar shift.  These can be broken into paraproduct operators for which the boundedness follows by the $\rm BMO$ assumption.  The main novelty in this proof is that we actually demonstrate a more general result by showing that a version of the above Theorem holds in product spaces of homogeneous type $X_1\times X_2$ in terms of the product BMO space $\bmo(X_1\times X_2)$ (for the definition, we refer to Section 2, see also Definition 2.6 in \cite{DLWY}).  We provide a statement of the main result in this direction as follows,
which will be proved in Section \ref{sec:representation}.
\begin{thm}\label{UpperBdThm}
Let $(X_i,\rho_i, \mu_i)$ be a space of homogeneous type.  Let $T_i$ be the Calder\'on--Zygmund operator on $X_i$ and let $b\in \bmo(X_1\times X_2)$. Then we have
\begin{align*}
\| [[b, T_1], T_2] \|_{L^2(X_1\times X_2, \mu_1\times\mu_2)\to L^2(X_1\times X_2, \mu_1\times\mu_2)} \leq C \|b\|_{\bmo(X_1\times X_2)}.
\end{align*}
\end{thm}
For precise definitions of the product spaces of homogeneous type, the product BMO space, and Calder\'on--Zygmund operators, we refer to Section 2,  see also \cite{HLW}.  Since we have that $\mathbb{R}_\lambda$ is a space of homogenous type, it is clear that Theorem \ref{thm main1} follows from the above theorem as a corollary.

%
%
%
%

\bigskip

The second main result of this paper is characterization of the little bmo space associated with $\Delta_\lz$, $\lbmo(\rlz)$, which is the space of functions satisfying the following definition.

\begin{defn}
A function $b\in L^1_{loc}(\rlz)$ is in $\lbmo(\rlz)$ if
\begin{align}
\|b\|_{\lbmo(\rlz)}:= \sup_{R\subset \R_+\times\R_+} \frac1{\mu_\lz(R)}\iint_R |b(x_1,x_2)-m_R(b)|d\mu_\lz(x_1,x_2) <\infty,
\end{align}
where
\begin{equation}\label{mean valu rect}
 m_R(b) := \frac1{\mu_\lz(R)}\iint_R b(x_1,x_2) d\mu_\lz(x_1,x_2)
\end{equation}
is the mean value of $b$ over the rectangle $R$.
\end{defn}
One can easily observe that this norm is equivalent to the following norm:
$$
\left\Vert b\right\Vert_{ \lbmo(\mathbb{R}_\lambda)} \approx \max \left\{\sup_{x\in\mathbb{R}_+} \left\Vert b(x,\cdot)\right\Vert_{\bmo_{\Delta_{\lambda}}(\mathbb{R}_{+},\dmz)}, \sup_{y\in\mathbb{R}_+} \left\Vert b(\cdot,y)\right\Vert_{\bmo_{\Delta_{\lambda}}(\mathbb{R}_{+},\dmz)}\right\};
$$
namely these functions are uniformly in $\bmo_{\Delta_{\lambda}}(\mathbb{R}_{+},\dmz)$ in each variable separately.  This leads to the following characterization of $\lbmo(\rlz)$:

\begin{thm}\label{thm main2}
\label{1.1}
Let $b\in  L^2_{loc}(\rlz)$. The following conditions are equivalent:
\begin{enumerate}
\item[(i)] $b\in $ \lbmo$(\rlz)$;
\item[(ii)] The commutators $[b,R_{\Delta_\lambda,1}]$ and $[b,R_{\Delta_\lambda,2}]$ are both bounded on $L^2(\rlz)$;
\item[(iii)] There exist $f_1,f_2,g_1,g_2\in L^\infty(\rlz)$ such that $b=f_1+\rizo g_1=f_2+\rizt g_2$ and moreover,
$\|b\|_{\lbmo(\rlz)}\approx \inf\Big\{ \max_{i=1,2} \big\{\|f_i\|_{L^\infty(\rlz)},\|g_i\|_{L^\infty(\rlz)}\big\} \Big\}$, where the infimum is taken over all possible decompositions of
$b$;
\item[(iv)] The commutator $[b,R_{\Delta_\lambda,1}R_{\Delta_\lambda,2}]$ is bounded on $L^2(\rlz)$.
\end{enumerate}
\end{thm}
The proof of the equivalence between (i) and (ii) in this theorem, relies on a recent new result obtained by a subset of authors in \cite{DLWY}, which shows that in the one parameter setting $b\in \bmo(\R_+,\dmz)$ if and only if the commutator
$[b, R_{\Delta_\lz}]$ is a bounded operator on $L^2(\R_+,\dmz)$.

Moreover, the proof of the equivalence between (i) and (iv) extends the result of Ferguson--Sadosky \cite{fs} to the Bessel setting,
where no analyticity or Fourier transform is available.
%
%
We prove this characterization by understanding a certain weak factorization of the predual of $\lbmo(\rlz)$. To obtain this, we first define the little Hardy space $h^{1,\infty}(\rlz)$ in terms of $(1,\infty)$-rectangular atoms with a one-parameter version of cancellation. However, it is less direct to see how the duality works by using only $(1,\infty)$-rectangular atoms. We also introduce the $(1,q)$-rectangular atoms for $1<q<\infty$, and then prove that $h^{1,\infty}(\rlz)$ can be characterised equivalently by $(1,q)$-rectangular atoms. Then, by using the $(1,2)$-rectangular atoms, the duality of $h^{1,\infty}(\rlz)$ with $\lbmo(\rlz)$ follows from the  standard argument, see for example \cite{CW2} (see also \cite[Section II, Chapter 3]{J}). This factorization particularly uses key estimates on the kernel of the Riesz transforms, especially the lower bound conditions, which was studied in \cite{bfbmt} and refined recently by the subset of authors \cite{DLWY}; these estimates are essentially different from the standard Riesz transforms on $\mathbb R^n$. We point out that the characterizations of the little Hardy space in terms of $(1,q)$-rectangular atoms are new even when we refer back to the classical case of Ferguson--Sadosky \cite{fs}.

Finally as a corollary of the characterization of $\lbmo(\rlz)$  in Theorem \ref{thm main2} and the Fefferman--Stein type decomposition of $\bmo(\rlz)$ as proved in \cite{DLWY2}, we show that:

\begin{cor}\label{coro} $\lbmo(\rlz)$ is a proper subspace of $\bmo_{\Delta_\lambda} (\rlz)$, i.e.,
$$\lbmo(\rlz)\subsetneq \bmo_{\Delta_\lambda} (\rlz).$$
\end{cor}
Again, this is in analogy with the corresponding results in the Euclidean setting.  Containment of the spaces follows from property (iii) and a similar characterization of product $\rm BMO$ in this setting.  The fact that it is a proper containment follows from a simple construction.  These results, as well as corollaries about the relevant factorizations, can be found in Section~\ref{sec:factorization}.

A natural question that arises from this work is whether the space $\bmo_{\Delta_\lambda} (\rlz)$ can be characterized by the iterated commutators:
$$
\|[ [b, R_{\Delta_\lambda,1}],R_{\Delta_\lambda,2}] \|_{L^2(\rlz)\to L^2(\rlz)} \approx  \|b\|_{\bmo_{\Delta_\lambda} (\rlz)}.
$$
As evidence for this we point out that in the case of one parameter this result was answered by a subset of the authors in \cite{DLWY}; and it was shown that the space $\bmo_{\Delta_\lambda} (\rlz)$ can indeed be characterized by the commutator.  We also point out that using the methods of Section \ref{sec:factorization} it is possible to obtain a lower bound on the iterated commutator in terms of a ``rectangle $\bmo_{\Delta_\lambda} (\rlz)$''.  While we would like to return to this characterization in subsequent work, we want to point out some challenges with obtaining the lower bound.  The analogous proof in the Euclidean spaces, \cite{fl,LPPW}, uses key properties of the Fourier transform, the Riesz/Hilbert transforms and wavelets.  Some of these tools do not translate well to the setting at hand and instead a new proof seems to be needed.



\section{Upper bound of  iterated commutator $[ [b, T_1],T_2] $}
\label{sec:representation}
\setcounter{equation}{0}

In this section we prove Theorem \ref{UpperBdThm}, which extends the main result of \cite{DO} to spaces of homogeneous type introduced by Coifman and Weiss \cite{CW2}.
We first recall some necessary notation and definitions on spaces of homogeneous type, including the product Calder\'on--Zygmund operators and product BMO space on space of homogeneous type as well as some fundamental tools such as the Haar basis and representation theorem, which will be crucial to the proof of Theorem \ref{UpperBdThm}.

%

\subsection{Preliminaries}
\label{sec:dyadic_systems}

%
By a quasi-metric we mean a mapping $\rho\colon X\times X\to
[0,\infty)$ that satisfies the axioms of a metric except for
the triangle inequality which is assumed in the weaker form
\begin{align}\label{quasimetric}
    \rho(x,y)
    \leq A_0(\rho(x,z)+\rho(z,y))
    \quad \text{for all $x,y,z\in X$}
\end{align}
with a constant $A_0\geq 1$.

We define the quasi-metric ball by $B(x,r) := \{y\in X: \rho(x,y)
< r\}$ for $x\in X$ and $r > 0$.
We say that a nonzero measure $\mu$ satisfies the
\emph{doubling condition} if there is a constant $C_\mu$ such
that for all $x\in\XX$ and $r > 0$,
\begin{equation}\label{eqn:doubling condition}
   \mu(B(x,2r))
   \leq C_\mu \mu(B(x,r))
   < \infty.
\end{equation}


We recall that $(X,d,\mu)$ is a space of homogeneous type in the sense of Coifman and Weiss \cite{CW2} if $d$ is a quasi-metric and $\mu$ is a nonzero measure satisfying the doubling condition.

We also denote the product
space
\begin{align}\label{X}
\X:= (X_1,d_1,\mu_1)\times (X_2,d_2,\mu_2),
\end{align}
where for each $i:=1,2$, the space $(X_i,d_i,\mu_i)$ is a space of homogeneous type, with
the coefficient $A_{0,i}$ for the quasi-metric $d_i$ as in \eqref{quasimetric} and with the
coefficient $C_{\mu_i}$  
for the measure $\mu_i$ as in  \eqref{eqn:doubling condition}, respectively.


We now recall the BMO and product BMO spaces on general spaces of homogeneous type.   The case of one parameter is the following, expected definition.

\begin{defn}
A locally integrable function $f$ is in $\bmo(X)$ if and only if
\begin{align}
\|f\|_{\bmo(X)}:={1\over\mu(B)}\int_B |f(x)-f_B|d\mu(x)<\infty,
\end{align}
where $f_B:= \mu(B)^{-1} \int_B f(y)d\mu(y)$, and $B$ is any quasi-metric ball in $X$.
\end{defn}

For the case of product BMO we need to introduce wavelets on spaces of homogeneous type.  To begin with, recall the set $\{x_\alpha^k\}$ of {\it reference dyadic
points} as follows. Let $\delta$ be a fixed small positive
parameter (for example, as noted in \cite[Section 2.2]{AuHyt},
it suffices to take $\delta\leq 10^{-3} A_0^{-10}$). For $k =
0$, let $\mathscr{X}^0 := \{x_\alpha^0\}_\alpha$ be a maximal
collection of 1-separated points in~$\XX$. Inductively, for
$k\in\mathbb{Z}_+$, let $\mathscr{X}^k := \{x_\alpha^k\}
\supseteq \mathscr{X}^{k-1}$  and $\mathscr{X}^{-k} :=
\{x_\alpha^{-k}\} \subseteq \mathscr{X}^{-(k-1)}$ be
$\delta^k$- and $\delta^{-k}$-separated collections in
$\mathscr{X}^{k-1}$ and $\mathscr{X}^{-(k-1)}$, respectively.

As shown in \cite[Lemma 2.1]{AuHyt}, for all $k\in\mathbb{Z}$ and
$x\in X$, the reference dyadic points satisfy
\begin{eqnarray}\label{delta sparse}
    d(x_\alpha^k,x_\beta^k)\geq\delta^k\ (\alpha\not=\beta),\hskip1cm
        d(x,\mathscr{X}^k)
    = \min_\alpha\,d(x,x_\alpha^k)
    < 2A_0\delta^k.
\end{eqnarray}
Also, taking $c_0 := 1$, $C_0 := 2A_0$ and $\delta \leq 10^{-3}
A_0^{-10}$, we see that $c_0$, $C_0$ and $\delta$ satisfy
$12A_0^3C_0\delta\leq c_0$ in \cite[Theorem 2.2]{HK}. By applying Hyt\"onen and Kairema's
construction (\cite[Theorem 2.2]{HK}).
We conclude
that there exists a set of dyadic cubes
$
    \{Q_\alpha^k\}_{k\in\mathbb{Z},\alpha\in\mathscr{X}^k}
$
associated with the reference dyadic points
$\{x_\alpha^k\}_{k\in\mathbb{Z},\alpha\in\mathscr{X}^k}$. We
call the reference dyadic point~$x_\alpha^k$ the \emph{center}
of the dyadic cube~$Q_\alpha^k$. We also identify with
$\mathscr{X}^k$ the set of indices~$\alpha$ corresponding to
$x_\alpha^k \in \mathscr{X}^{k}$. We now denote the system of dyadic cubes as
$$ \mathscr{D}:=\bigcup_k \mathscr{D}_k,\quad {\rm with}\  \mathscr{D}_k:=\{ Q_\alpha^k:\ \alpha\in \mathscr{X}^k \}. $$

Note that $\mathscr{X}^{k}\subseteq \mathscr{X}^{k+1}$ for
$k\in\mathbb{Z}$, so that every $x_\alpha^k$ is also a point of
the form $x_\beta^{k+1}$. We denote
$\mathscr{Y}^{k}:=\mathscr{X}^{k+1}\backslash \mathscr{X}^{k}$
and relabel the points $\{x_\alpha^k\}_\alpha$ that belong
to~$\mathscr{Y}^k$ as $\{y_\alpha^k\}_\alpha$.

\begin{defn}[\cite{HLW}]\label{def-BMO}
    We define the product BMO space $\bmo(\XX_1\times\XX_2)$ in
    terms of wavelet coefficients by
    $
        \bmo(\XX_1\times\XX_2)
        := \big\{ f \in (\GGp)' : \mathcal{C}(f)< \infty\},
    $
    with the quantity $\mathcal{C}(f)$ defined as follows:
    \begin{eqnarray}\label{Carleson norm}
        \mathcal{C}(f)
        := \sup_{\Omega}
            \Bigg\{ {1\over\mu(\Omega)}
            \sum_{\substack{R = Q_{\alpha_1}^{k_1}\times Q_{\alpha_2}^{k_2}
                \subset \Omega, \\
                k_1,k_2\in\mathbb{Z}, \alpha_1\in\mathscr{Y}^{k_1},
                \alpha_2\in\mathscr{Y}^{k_2} }}
            \big| \langle \psi_{\alpha_1}^{k_1}\psi_{\alpha_2}^{k_2}, f \rangle
            \big|^2 \Bigg\}^{1/2},
    \end{eqnarray}
    where $\Omega$ runs over all open sets in $\XX_1\times\XX_2$ with
    finite measure.
\end{defn}

Here we point out that the notation $(\GGp)' $ in the definition above denotes the space of distributions in the product setting $X_1\times X_2$. We recall the test function and distribution spaces, and the one-parameter version of which was defined by Han, M\"uller and Yang \cite{hmy06,hmy}, and then the product version by Han, Li and Lu \cite{HLL2}, where the extra reverse doubling conditions of the underlying measures are required. Here we cite the definition of test functions and distributions in both the one-parameter setting and product setting in \cite{HLW}, where there is no extra assumptions on the quasi-metric and doubling measure. Moreover, the notation $\psi_{\alpha}^{k}$, $\alpha\in \mathscr{Y}^{k_1}$, denotes the orthonormal basis on general spaces of homogeneous type $(X,d,\mu)$ constructed by Auscher and Hyt\"onen (see \cite{AuHyt} Theorem 7.1).


\smallskip

Next we recall the definition for Calder\'on--Zygmund operators on spaces of homogeneous type and the representation theorems for these Calder\'on--Zygmund operators.

A continuous function $K(x,y)$
defined on  $X \times X \backslash \{ (x, y): x=y\}$ is called a
{\it Calder\'on--Zygmund kernel} if there exist constant $C>0$ and a
regularity exponent $\varepsilon\in (0,1]$ such that
\begin{itemize}
\item[(a)] $|K(x,y)|\leq C V(x,y)^{-1}$;
\item[(b)] $\displaystyle\vert K(x,y)-K(x,y')\vert + \vert K(y,x)-K(y,x')\vert\leq C
\bigg(\frac{d(y,y')}{d(x,y)}\bigg)^{\varepsilon}V(x,y)^{-1}$
             \quad if\ \ $\displaystyle d(y,y')\le \frac{d(x,y)}{2A_0}$.
\end{itemize}
Above $V(x,y):=\mu(B(x,d(x,y))$.  The smallest such constant $C$ is denoted by $|K|_{CZ}.$ We say
that an operator $T$ is a {\it singular integral operator} associated with a Calder\'on--Zygmund kernel
$K$ if the operator $T$ is a continuous linear operator from
$C^\eta_0(X)$ into its dual such that
$$\langle Tf, g \rangle=\int_X\int_X g(x)K(x,y)f(y)d\mu(y)d\mu(x)$$
for all functions $f, g\in C^\eta_0(X)$ with disjoint supports. Here
$C^\eta_0(X)$ is the space of all continuous functions on $X$ with compact support such that
$$ \|f\|_{C_0^\eta(X)}:=\sup_{x\not=y} {|f(x)-f(y)|\over d(x,y)^\eta} <\infty.$$
The operator $T$ is said to be a {\it Calder\'on--Zygmund operator} if it
extends to be a bounded operator on $L^2(X).$ If $T$ is a
Calder\'on--Zygmund operator associated with a kernel $K$, its
operator norm is defined by $\|T\|_{CZ}=\|T\|_{L^2\rightarrow
L^2}+ | K|_{CZ}$.

We now recall the explicit construction in \cite{KLPW} of a Haar basis
$\{h^u_Q\colon Q\in\D, u = 1,\dots,M_Q - 1\}$ for $L^p(X)$,
$1 < p < \infty$, associated to the dyadic cubes
$Q\in\mathscr{D}$ as follows. Here $M_Q := \#\ch(Q) = \# \{R\in
\mathscr{D}_{k+1}\colon R\subseteq Q\}$ denotes the number of
dyadic sub-cubes (``children'') the cube $Q\in \mathscr{D}_k$
has. 

\begin{thm}[\cite{KLPW}]\label{thm:convergence}
    Let $(X,\rho)$ be a geometrically doubling quasi-metric
    space and suppose $\mu$ is a positive Borel measure on $X$
    with the property that $\mu(B) < \infty$ for all balls
    $B\subseteq X$. For $1 < p < \infty$, for each $f\in
    L^p(X)$, we have
    \[
        f(x)
        =  \sum_{Q\in\mathscr{D}}\sum_{u=1}^{M_Q-1}
            \langle f,h^u_Q\rangle_{L^2(X)} h^u_Q(x), 
    \]
    where the sum converges (unconditionally) both in the
    $L^p(X)$-norm and pointwise $\mu$-almost everywhere.
 \end{thm}

We now recall the decomposition of  a Calder\'on--Zygmund operator $T$ into dyadic Haar shifts, (see for example \cite{Hy,NRV,NV}).

\begin{thm}\label{thm repre}
Let $T$ be a Calder\'on--Zygmund operator associated with a kernel $K$. Then it has a decomposition:
for $f,g \in C_0^\eta(X)$,
\begin{align}
\langle g,Tf\rangle_{L^2(X)} = c(\|T\|_{2\to2} +|K|_{CZ}) E_w \sum_{m,\, n=0}^\infty \tau(m, n) \langle g,S^{m,\,n}_w f\rangle_{L^2(X)},
\end{align}
where $E_w$ is the expectation operator with respect to the random variable $w$, $\mathscr{D}_w$ is the random dyadic system, $S_w^{m,\,n}$ is a dyadic Haar shift with parameters $m,\, n$ on  $\mathscr{D}_w$ defined as follows
$$ S_w^{m,\,n}(f)(x) = \sum_{L\in \mathscr{D}_w}
\sum_{\substack{I\in \mathscr{D}_w,\, I\subset L\\ g(I)=g(L)+m }}\sum_{i=1}^{M_I-1}
\sum_{\substack{J\in \mathscr{D}_w,\, J\subset L\\ g(J)=g(L)+n }} \sum_{j=1}^{M_J-1}
a_{L,\,I,\,J} \langle  h_I^i, f \rangle_{L^2(X)} h_J^j(x) $$
with
$$ |a_{L,\,I,\,J}|\leq {\frac{\sqrt{\mu(I)}\sqrt{\mu(J)}}{\mu(L)}}\quad{\rm and}\quad \tau(m,n) \leq C\delta^{m+n},  $$
where $\delta$ is the small positive number satisfying $\delta\leq 10^{-3} A_0^{-10}$ with $A_0$ the constant in \eqref{quasimetric}.
\end{thm}

With these tools at hand, we note that the idea and approach of the proof of Theorem \ref{UpperBdThm} is similar to the main result of \cite{DO}. For the sake of clarity, we provide an outline of the proof in the following two subsections.

\subsection{The one parameter case: $[b,T]$, $b\in {\rm BMO}(X)$}

To begin with, we derive a decomposition of the one-parameter commutator $[b,T]$ into basic paraproduct type operators.
\begin{thm}\label{one-para decomp}
Let $b\in\text{BMO}(X)$, $f\in C_0^\eta(X)$, and $T$ be a Calder\'on--Zygmund operator. Then,

{\bf (i)} for a cancellative dyadic shift $S_\omega^{m,\,n}$, $[b, S^{m,\,n}]$  can be represented as a finite linear combination of operators of the form
\begin{equation}\label{resp parapro-1}
S_\omega^{m,\,n}(B_k(b,f)),\, B_k(b, S_\omega^{m,\,n}f)
\end{equation}
where $k\in\mathbb{Z}, 0\le k\le \max(m, n)$ and the total number of terms
is bounded by $C(1+\max(m,n))$  for some universal constant $C$;

\medskip
{\bf (ii)} for a noncancellative dyadic shift $S_\omega^{0,\,0}$ with symbol $a$, $[b, S_\omega^{0,\,0}]f$ can be represented as a finite linear combination of operators of the form
\begin{equation}\label{resp parapro-2}
S_\omega^{0,\,0}(B_0(b, f)), B_0(b, S_\omega^{0,\,0}f), \tilde{B}_0(b, S_\omega^{0,\,0}f), P(b, a,f), P^\ast(b,a,f)
\end{equation}
and the total number of terms is bounded by a universal constant.
\end{thm}

The paraproduct like operators in the above theorem are defined as the following. The generalized dyadic paraproduct
\begin{equation}\label{B0}
B_k(b, f):=\sum_{I}\sum_{i'=1}^{M_{I^{(k)}}-1}\sum_{i=1}^{M_{I}-1}\langle b, h^{i'}_{I^{(k)}}\rangle_{L^2(X)}\langle f, h_I^i\rangle_{L^2(X)} h_I^i \, h^{i'}_{I^{(k)}},
\end{equation}
where $I^{(k)}$ denotes the $k$-th dyadic ancestor of $I$. Observe that when $k=0$, this is the classical paraproduct
\begin{equation}\label{B0t}
\tilde{B}_0(b, f):=\sum_{I}\sum_{i=1}^{M_{I}-1}\langle b, h^{i}_{I}\rangle_{L^2(X)}\langle f, h_I^0\rangle_{L^2(X)} h_I^0 \, h^{i}_{I}.
\end{equation}
And the trilinear operator
\begin{equation}\label{P}
P(b, a,f):=\sum_{I}\sum_{i=1}^{M_I-1}\langle b, h_I^i\rangle_{L^2(X)}\langle f, h_I^i\rangle_{L^2(X)} h^i_I h^{i}_I \sum_{J,\,J\subsetneq I}\sum_{j=1}^{M_J-1} \langle a, h_J^j\rangle_{L^2(X)} h_J^j,
\end{equation}
with $P^\ast$ being understood as the adjoint of $P$ with $b$ and $a$ fixed. The important property of the above operators is that they are uniformly bounded on $L^2$ with BMO symbols.
\begin{lem}\label{lem Bk P}
Given $a, b\in {\rm BMO}(X)$ and $k\geq0$, we have
$$ \|B_k(b, f)  \|_{L^2(X)} \lesssim \|b\|_{{\rm BMO}(X)} \|f\|_{L^2(X)}, $$
$$ \|\tilde{B}_0(b, f)  \|_{L^2(X)} \lesssim \|b\|_{{\rm BMO}(X)} \|f\|_{L^2(X)}, $$
and
$$ \|P(a, b, f)  \|_{L^2(X)} \lesssim \|a\|_{{\rm BMO}(X)}\|b\|_{{\rm BMO}(X)} \|f\|_{L^2(X)}.$$
\end{lem}
The lemma is well-known for $\tilde{B}_0(b,f)$, which is the classical paraproduct. For $B_k(b,f)$, $k\geq1$
and for $P(b, a,f)$, the boundedness follows from adaptations and modifications of \cite[Lemma 3.6 and 3.7]{DO} to the spaces of homogeneous type. The relevant properties of spaces of homogeneous type here are the orthogonality of the Haar bases $\{h_J^j\}_{J,j}$, $H^1$-BMO duality, dyadic square function characterization of dyadic $H^1$ and the John-Nirenberg inequality.

The proof of Theorem \ref{one-para decomp} follows essentially the same strategy of \cite[Theorem 3.2]{DO}. Unlike the Euclidean setting, where associated with any $Q\in\mathcal{D}$ are a fixed number of Haar functions that are constant on each child (of the same measure) of $Q$, in spaces of homogeneous type, there are $M_Q$ Haar functions $h_Q^u$ for any $Q\in\mathcal{D}$ and the measure of each child of $Q$ can be different. Fortunately, by closely examining the argument in \cite{DO}, one observes without much difficulty that the only properties of the Haar systems it relies on are the martingale structure: 
\begin{align}
\sum_{J:\ I\subsetneq J} \sum_{j=1}^{M_J-1} \langle f, h_J^j\rangle_{L^2(X,\mu)} h_J^j h_I^i
= \langle f,h_I^0\rangle_{L^2(X,\mu)} h_I^0h_I^i.
\end{align}
 and the fact that the dyadic cubes in $\mathcal{D}$ are properly nested. We omit the details of the proof.

In particular, Theorem \ref{one-para decomp}, together with Lemma \ref{lem Bk P} and the representation of Calder\'on-Zygmund operators by Haar shifts (Theorem \ref{thm repre}), implies almost immediately the upper bound of the commutator $[b,T]$ in spaces of homogeneous type:
\begin{align}\label{upper bound one-para}
\| [b, T] \|_{L^2(X)\to L^2(X)} \leq C \|b\|_{BMO(X)},
\end{align}
which recovers the upper bound result of \cite{crw,kl,bc}. More importantly, Theorem \ref{UpperBdThm} follows from iterating Theorem \ref{one-para decomp} and BMO estimates of certain bi-parameter paraproduct like operators, which we explain in the next subsection.

\subsection{The iterated case: $[[b,T_1],T_2]$}
Applying the representation theorem (Theorem \ref{thm repre}) in both variables, one could obtain Theorem \ref{UpperBdThm} by proving for any $f\in C^\eta_0(X_1\times X_2)$ that
\begin{align}\label{upp bdd commu bi}
&\bigg\|\sum_{m_1,\,m_2,\,n_1,\,n_2=0}^\infty \tau(m_1,n_1)\tau(m_2,n_2) [[b, S_1^{m_1,n_1}], S_2^{m_2,n_2}]f\bigg\|_{L^2(X_1\times X_2)}\\
&\quad\ls \|b\|_{BMO(X_1\times X_2)}\|f\|_{L^2(X_1\times X_2)}.\nonumber
\end{align}
By an iteration of Theorem \ref{one-para decomp}, one can represent $[[b, S_1^{m_1,\,n_1}], S_2^{m_2,\,n_2}]$ as a finite linear combination of basic operators which are essentially tensor products of the operators $B_k$, $\tilde{B}_0$ and $P$ in the one-parameter setting as in \eqref{B0}, \eqref{B0t}  and \eqref{P}, and the total number of terms is no greater than $C(1+\max (m_1,n_1))(1+\max (m_2,n_2))$. Estimate (\ref{upp bdd commu bi}) then follows from the uniform boundedness of these operators which we conclude in Lemma \ref{lem bipara} below. The proof of Theorem \ref{UpperBdThm} is thus complete.

More precisely, we need to consider the following paraproduct like operators in the bi-parameter setting (to condense notation that we omit the subscript $L^2(X_1\times X_2)$ on the inner products). To begin with, we let $a, b\in {\rm BMO}(X_1\times X_2)$, $a^1\in {\rm BMO}(X_1)$ and $a^2\in {\rm BMO}(X_2)$. The generalized bi-parameter dyadic paraproduct
\begin{equation*}
B_{k,l}(b, f):=
\sum_{I}\sum_{i'=1}^{M_{I^{(k)}}-1}\sum_{i=1}^{M_{I}-1}\sum_{J}\sum_{j'=1}^{M_{J^{(l)}}-1}\sum_{j=1}^{M_{J}-1}\langle b, h^{i'}_{I^{(k)}}\otimes h^{j'}_{J^{(l)}}\rangle\langle f, h_I^i\otimes h_J^j\rangle h_I^i \, h^{i'}_{I^{(k)}}\otimes h_J^j\, h^{j'}_{J^{(l)}}.
\end{equation*}
Parallel to \eqref{B0t}, we also have
\begin{align*}
\tilde{B}_{k,l}^{(1)}(b, f)&:=
\sum_{I}\sum_{i'=1}^{M_{I^{(k)}}-1}\sum_{J}\sum_{j'=1}^{M_{J^{(l)}}-1}\sum_{j=1}^{M_{J}-1}\langle b, h^{i'}_{I^{(k)}}\otimes h^{j'}_{J^{(l)}}\rangle\langle f, h_I^0\otimes h_J^j\rangle h_I^0 \, h^{i'}_{I^{(k)}}\otimes h_J^j\, h^{j'}_{J^{(l)}},\\
\tilde{B}_{k,l}^{(2)}(b, f)&:=
\sum_{I}\sum_{i'=1}^{M_{I^{(k)}}-1}\sum_{i=1}^{M_{I}-1}\sum_{J}\sum_{j'=1}^{M_{J^{(l)}}-1}\langle b, h^{i'}_{I^{(k)}}\otimes h^{j'}_{J^{(l)}}\rangle\langle f, h_I^i\otimes h_J^0\rangle h_I^i \, h^{i'}_{I^{(k)}}\otimes h_J^0\, h^{j'}_{J^{(l)}},\\
\tilde{B}_{k,l}^{(3)}(b, f)&:=
\sum_{I}\sum_{i'=1}^{M_{I^{(k)}}-1}\sum_{J}\sum_{j'=1}^{M_{J^{(l)}}-1}\langle b, h^{i'}_{I^{(k)}}\otimes h^{j'}_{J^{(l)}}\rangle\langle f, h_I^0\otimes h_J^0\rangle h_I^0 \, h^{i'}_{I^{(k)}}\otimes h_J^0\, h^{j'}_{J^{(l)}}.
\end{align*}
The trilinear operator
\begin{equation*}
\begin{split}
PP(b, a,f):=&\sum_{I}\sum_{i=1}^{M_I-1}\sum_{J}\sum_{j=1}^{M_J-1}\langle b, h_I^i\otimes h_J^j\rangle\langle f, h_I^i\otimes h_{J}^j\rangle h^i_I h^{i}_I\otimes h^j_J h^j_J \cdot\\
&\sum_{I_1:\,I_1\subsetneq I}\sum_{i'=1}^{M_{I_1}-1}\sum_{J_1:\,J_1\subsetneq J}\sum_{j'=1}^{M_{J_1}-1}\langle a, h_{I_1}^{i'}\otimes h_{J_1}^{j'}\rangle h_{I_1}^{i'}\otimes h_{J_1}^{j'},
\end{split}
\end{equation*}
where all the Haar functions are cancellative. And the new mixed type trilinear operators
\begin{align*}
&\begin{split}
BP_k(b,a^2,f):=&\sum_{I}\sum_{i'=1}^{M_{I^{(k)}}-1}\sum_{i=1}^{M_{I}-1}\sum_{J}\sum_{j=1}^{M_J-1}\langle b, h^{i'}_{I^{(k)}}\otimes h_J^j\rangle \langle f, h_I^i\otimes h_J^j\rangle h_I^i h^{i'}_{I^{(k)}}\otimes h^j_J h^j_J \cdot\\
&\sum_{J_1:\,J_1\subsetneq J}\sum_{j'=1}^{M_{J_1}-1}\langle a^2, h_{J_1}^{j'}\rangle_2 h_{J_1}^{j'},
\end{split}\\
&\begin{split}
\tilde{B}P_k(b,a^2,f):=&\sum_{I}\sum_{i'=1}^{M_{I^{(k)}}-1}\sum_{J}\sum_{j=1}^{M_J-1}\langle b, h^{i'}_{I^{(k)}}\otimes h_J^j\rangle \langle f, h_I^0\otimes h_J^j\rangle h_I^0 h^{i'}_{I^{(k)}}\otimes h^j_J h^j_J \cdot\\
&\sum_{J_1:\,J_1\subsetneq J}\sum_{j'=1}^{M_{J_1}-1}\langle a^2, h_{J_1}^{j'}\rangle_2 h_{J_1}^{j'},
\end{split}\\
&\begin{split}
PB_l(b,a^1,f):=&\sum_{I}\sum_{i=1}^{M_I-1}\sum_{J}\sum_{j'=1}^{M_{J^{(l)}}-1}\sum_{j=1}^{M_{J}-1}\langle b, h_I^i\otimes h^{j'}_{J^{(l)}}\rangle\langle f, h_{I}^i\otimes h_J^j\rangle h_I^i h_I^i\otimes h_J^j\, h^{j'}_{J^{(l)}} \cdot\\
&\sum_{I_1:\,I_1\subsetneq I}\sum_{i'=1}^{M_{I_1}-1} \langle a^1, h_{I_1}^{i'}\rangle_1 h_{I_1}^{i'},
\end{split}\\
&\begin{split}
P\tilde{B}_l(b,a^1,f):=&\sum_{I}\sum_{i=1}^{M_I-1}\sum_{J}\sum_{j'=1}^{M_{J^{(l)}}-1}\langle b, h_I^i\otimes h^{j'}_{J^{(l)}}\rangle\langle f, h_{I}^i\otimes h_J^0\rangle h_I^i h_I^i\otimes h_J^0\, h^{j'}_{J^{(l)}} \cdot\\
&\sum_{I_1:\,I_1\subsetneq I}\sum_{i'=1}^{M_{I_1}-1} \langle a^1, h_{I_1}^{i'}\rangle_1 h_{I_1}^{i'}.
\end{split}
\end{align*}
\begin{lem}\label{lem bipara}
Given $a, b\in {\rm BMO}(X_1\times X_2)$, $a^1\in {\rm BMO}(X_1)$ and $a^2\in {\rm BMO}(X_2)$, we have
\[
\|PP(b,a,f) \|_{L^2(X_1\times X_2)} \lesssim \|b\|_{{\rm BMO}(X_1\times X_2)} \|a\|_{{\rm BMO}(X_1\times X_2)}\|f\|_{L^2(X_1\times X_2)}
\]
and the same for $PP_1(b,a,f)$, which denotes the partial adjoint of $PP$ in the first variable with respect to the third input function;
moreover, for $k, l\geq0$, we have
\[
\|B_{k,l}(b, f)  \|_{L^2(X_1\times X_2)} \lesssim \|b\|_{{\rm BMO}(X_1\times X_2)} \|f\|_{L^2(X_1\times X_2)}
\]
and the same for $\tilde{B}^{(1)}_{k,l}(b, f)$, $\tilde{B}^{(2)}_{k,l}(b, f)$ and $\tilde{B}^{(3)}_{k,l}(b, f)$;
\[
 \|BP_k(b,a^2,f) \|_{L^2(X_1\times X_2)} \lesssim \|b\|_{{\rm BMO}(X_1\times X_2)} \|a^2\|_{{\rm BMO}(X_2)}\|f\|_{L^2(X_1\times X_2)}
\]
and the same for $\tilde{B}P_k(b,a^2,f)$;
\[
\|PB_l(b,a^1,f) \|_{L^2(X_1\times X_2)} \lesssim \|b\|_{{\rm BMO}(X_1\times X_2)} \|a^1\|_{{\rm BMO}(X_1)}\|f\|_{L^2(X_1\times X_2)}
\]
and the same for $P\tilde{B}_l(b,a^1,f)$.
\end{lem}
The above result can be derived similarly as in \cite[Lemmas 4.1, 4.2, and 4.5]{DO},
therefore we omit most of the details. We point out that a key fact that is crucial is
the following multi-parameter John-Nirenberg inequality in the homogeneous setting. The multiparameter John-Nirenberg inequality
was first introduced in \cite[Section III]{CF} for the product BMO space defined via the wavelet basis (see also \cite[Proposition 4.1]{Tao} for dyadic product BMO on $\R\times \R$ defined via Haar basis). We note that
this John-Nirenberg inequality also holds with the Haar system in the setting of space of homogeneous type. For the details, we refer to
\cite[pp.199--200]{CF} and omit it here.
\begin{lem}\label{lem bipara-2}
Given $b\in {\rm BMO}(X_1\times X_2)$ and $p\in(1,\infty)$, there holds
\begin{align}
\bigg\| \bigg( \sum_{R=I\times J\subset\Omega} \sum_{i=1}^{M_I-1}\sum_{j=1}^{M_j-1}\big| \langle b, h_I^ih_J^j \rangle\big|^2 {\frac{\chi_R}{\mu(R)} } \bigg)^{1/2}\bigg\|_{L^p(X_1\times X_2)} \leq C\|b\|_{{\rm BMO}(X_1\times X_2)} \mu(\Omega)^{1/p}.
\end{align}
\end{lem}

\section{Proof of Theorem \ref{thm main2}}
\label{sec:factorization}
\setcounter{equation}{0}

\subsection{Proof of (i)$\Longleftrightarrow$(ii)}

Suppose that $b\in \rm{bmo}(\rlz)$. Then we know that
for any fixed $x_2\in\R$, $b(x_1,x_2)$ as a function of $x_1$ is in the standard one-parameter BMO$(\R_+,\dmz)$, a symmetric result holds for the roles of $x_1$ and $x_2$ interchanged. Moreover, we further have that
\begin{align}\label{bmo norm equiv}
\|b\|_{\lbmo(\rlz)} \approx \sup_{x_1\in\R_+} \|b(x_1,\cdot)\|_{\bmo(\R_+,\dmz)} +
 \sup_{x_2\in\R_+} \|b(\cdot,x_2)\|_{\bmo(\R_+,\dmz)},
\end{align}
where the implicit constants are independent of the function $b$.

Next, we recall a recent result by a subset of the authors \cite{DLWY}, where they obtained that
\begin{align}\label{BMO norm equiv}
\|b\|_{\bmo(\R_+,\dmz)} \approx \|[b,\riz]\|_{L^2(\R_+,\dmz) \to L^2(\R_+,\dmz)},
\end{align}
where BMO$(\R_+,\dmz)$ is the standard one-parameter BMO space on $(\R_+,\dmz)$.

Combining \eqref{bmo norm equiv} and  \eqref{BMO norm equiv}, we obtain that
\begin{align*}
\|b\|_{\rm bmo(\rlz)} &\approx \sup_{x_1\in\R_+}    \|[b(x_1,\cdot),\rizt]\|_{L^2(\R_+,\dmz) \to L^2(\R_+,\dmz)} \\
&\hskip1cm   +
 \sup_{x_2\in\R_+}  \|[b(\cdot,x_2),\rizo]\|_{L^2(\R_+,\dmz) \to L^2(\R_+,\dmz)},
 \end{align*}
which implies that (i)$\Longleftrightarrow$(ii).

\subsection{Proof of (i)$\Longleftrightarrow$(iii)}

From \cite{bdt}, we know that $H^1(\R_+,\dmz)$ can be characterized via Bessel Riesz transforms, i.e.,
$f\in H^1(\R_+,\dmz)$ if and only if $f, \riz (f) \in L^1(\R_+,\dmz)$, and
$$ \|f\|_{H^1(\R_+,\dmz)}\approx \|f\|_{L^1(\R_+,\dmz)}+\|\riz (f)\|_{L^1(\R_+,\dmz)}.  $$
Then by the duality of $H^1(\R_+,\dmz)$ with $\bmo(\R_+,\dmz)$, and following the same approach as in \cite{fs}, we obtain the following decomposition for
$\bmo(\R_+,\dmz)$:

\smallskip
$b\in \bmo(\R_+,\dmz)$ if and only if there exist $f,g\in L^\infty(\R_+,\dmz)$ such that
\begin{align}\label{BMO decom}
b=f+\riz g.
\end{align}
Moreover,
$$\|b\|_{{\bmo(\R_+,\dmz)}}\approx \inf \{ \|f\|_{L^\infty(\R_+,\dmz)}+\|g\|_{L^\infty(\R_+,\dmz)} \},$$
where the infimum is taken over all possible decompositions of $b$ as in \eqref{BMO decom}. As a consequence, the argument (i)$\Longleftrightarrow$(iii)
follows from \eqref{bmo norm equiv} and \eqref{BMO decom}.

\subsection{Proof of (i)$\Longleftrightarrow$(iv)}

\subsubsection{Proof of (i)$\Longrightarrow$(iv)}

We point out that the proof of the upper bound of $[b,\rizo\rizt]$ follows directly from the property of bmo$(\rlz)$ and the $L^2$ boundedness of the Bessel Riesz transforms $\rizo$ and $\rizt$.

To see this, for $b\in {\rm bmo}(\rlz)$, we remark that
 $$[b, \rizo\rizt]= \rizo[b,\rizt] + [b,\rizo]\rizt.$$
 Then based on \eqref{bmo norm equiv} and the result of  \cite{DLWY}, we know that
\begin{align*}
&\big\|[b,\rizt] \big\|_{L^2(\rlz)\to L^2(\rlz)} + \big\| [b,\rizo]\big\|_{L^2(\rlz)\to L^2(\rlz)} \\
&\quad\ls \sup_{x_1\in\R_+} \|b(x_1,\cdot)\|_{\bmo(\R_+,\dmz)} +
 \sup_{x_2\in\R_+} \|b(\cdot,x_2)\|_{\bmo(\R_+,\dmz)}\\
&\quad\ls \|b\|_{{\rm bmo}(\rlz)}.
\end{align*}

Then, denote by ${\rm Id}_1$ and ${\rm Id}_2$ the identity operator on $L^2(\R_+,\dmz)$ for the first and second variable, respectively. We
further have
$$[b, \rizo\rizt]= (\rizo\otimes {\rm Id}_2) \circ   [b,\rizt] + [b,\rizo] \circ ({\rm Id}_1\otimes \rizt),$$
where we use $T_1 \circ T_2$ to denote the composition of two operators $T_1$ and $T_2$. Thus, we obtain that
\begin{align*}
&\big\| [  b, \rizo\rizt ] \big\|_{L^2(\rlz)\to L^2(\rlz)} \\
&\quad=\big\| (\rizo\otimes {\rm Id}_2) \circ   [b,\rizt] + [b,\rizo] \circ ({\rm Id}_1\otimes \rizt) \big\|_{L^2(\rlz)\to L^2(\rlz)}\\
& \quad\leq \big\| (\rizo\otimes {\rm Id}_2) \circ  [b,\rizt]  \big\|_{L^2(\rlz)\to L^2(\rlz)}  + \big\|[b,\rizo]\circ ({\rm Id}_1\otimes \rizt)  \big\|_{L^2(\rlz)\to L^2(\rlz)}\\
& \quad\leq   \big\| \rizo\|_{L^2(\rlz)\to L^2(\rlz)} \big\|[b,\rizt] \big\|_{L^2(\rlz)\to L^2(\rlz)} \\
&\quad\quad+ \big\| [b,\rizo]\big\|_{L^2(\rlz)\to L^2(\rlz)}\big\|\rizt \big\|_{L^2(\rlz)\to L^2(\rlz)}\\
&\quad\ls \|b\|_{{\rm bmo}(\rlz)},
\end{align*}
which implies (i)$\Longrightarrow$(iv).

\subsubsection{Proof of (i)$\Longleftarrow$(iv)}

We begin with some preliminaries.
\begin{prop}[\cite{DLWY}]
\label{p-estimate of riesz kernel}
The Riesz kernel $\riz(x,y)$ satisfies:
\begin{itemize}
  \item [(i)] There exist $K_1>2$ large enough and a positive constant $C_{K_1,\,\lz}$ such that
  for any $x,\,y\in\mathbb{R}_+$ with $y>K_1x$,
  \begin{equation}\label{lower bound of riesz kernel}
  \riz(x, y)\ge C_{K_1,\,\lz}\frac x{y^{2\lz+2}}.
  \end{equation}
  \item [(ii)]There exist $K_2\in(0, 1)$ small enough and a positive constant $C_{K_2,\,\lz}$ such that
  for any $x,\,y\in\mathbb{R}_+$ with $y<K_2x$,
  \begin{equation}\label{upper bound of riesz kernel}
  \riz(x, y)\le -C_{K_2,\,\lz}\frac1{x^{2\lz+1}}.
  \end{equation}
\item [(iii)] There exist ${K}_3\in(0, 1/2)$  small enough and a positive constant $C_{{K}_3,\lz}$
such that for any $x,\,y\in\mathbb{R}_+$ with $0<y/x-1<{K}_3$,
\begin{eqnarray*}
\riz(x,y)
&\geq& C_{{K}_3,\lz}{1\over x^\lambda y^\lambda}{1\over y-x}.
\end{eqnarray*}
\end{itemize}
\end{prop}

\begin{defn}\label{def atom}
Suppose $q\in(1,\infty]$. A $q$-atom on $\rlz$ is a function
$a\in L^q(\rlz)$ supported on a rectangle $R \subset \rlz$ with $\|a\|_{L^q(\rlz)} \leq \mu_\lz(R)^{{1\over q}-1}$ and satisfying the cancellation property
$$  \int_{\R_+\times\R_+} a(x_1,x_2)d\mu_\lz(x_1,x_2)=0. $$
Let $Atom_q(\rlz)$ denote the collection of all such atoms.
\end{defn}
\begin{defn}\label{def little h1}
Suppose $q\in(1,\infty]$. The atomic Hardy space $h^{1,q}(\rlz)$ is defined as the set of functions of the form
\begin{align}\label{atomic decomposition}
f=\sum_i\alpha_i a_i
\end{align}
with $\{a_i\}_i\subset Atom_q(\rlz)$, $\{\alpha_i\}_i \subset \mathbb{C}$ and $\sum_i|\alpha_i| <\infty$.
Moreover, $h^{1,q}(\rlz)$  is equipped with the norm $\|f\|_{h^{1,q}(\rlz)}:= \inf \sum_i|\alpha_i|$ where
the infimum is taken over all possible decompositions of $f$ in the form \eqref{atomic decomposition}.
\end{defn}

For these little Hardy spaces, we first have
 the following conclusion.
\begin{thm}\label{t-atm Hard q indepen}
Let $q\in(1, \fz)$.
Then the spaces $\nhoq$ and $h^{1,\,\fz}(\rlz)$
coincide with equivalent norms.
\end{thm}

We first recall the following  Whitney type covering lemma from \cite{CW2}.

\begin{lemma}\label{l-Whitney type covering lemma} Suppose
$U\subsetneqq \R_+\times\R_+$ is an open bounded set and $\wz C\in[1,\fz)$.
Then there exists a sequence of cubes $\{Q_j\}_j$ satisfying
\begin{itemize}
  \item [{\rm (i)}] $U=\cup_j Q_j=\cup_j{\wz C}Q_j$;
  \item [{\rm (ii)}] there exists a positive constant $M$
  such that no point of $\R_+\times\R_+$ belongs to
  more than $M$ of the balls ${\wz C}Q_j$,
which is called as the \emph{$M$-disjointness} of $\{{\wz C}Q_j\}_j$;
  \item [{\rm (iii)}] $3{\wz C}Q_j\cap(\R_+\times\R_+\setminus U)
  \not=\emptyset$ for each $j$.
\end{itemize}
\end{lemma}

Now we establish a useful lemma which is a variant of \cite[Lemma (3.9)]{CW2}.
To this end, we recall the strong maximal function defined by setting,
for all $(x_1,x_2)\in \R_+\times\R_+$,
$$
\cm_sf(x_1,x_2):=\displaystyle\sup_{R\ni (x_1,x_2)}
\frac1{\mu_\lambda(R)}\int_{R}|f(y_1,y_2)|\,d\mu_\lambda(y_1,y_2).
$$
It is already known that $\cm_s$ is bounded on $\lpzd$, with $p\in(1,\fz)$.

\begin{lemma}\label{l7.5}
If $f\in L^1_{\loc}(\mu_\lambda)$ has support in $R_0:=I_0\times J_0$ centered at $(x^1_0, x_0^2)$, then there exists a positive constant $C_1$ such that
$$
U^\alpha:=\{(x_1, x_2)\in\rlz:\ \cm_sf(x_1, x_2)>\az\}\subset 3R_0
$$
whenever $\az\in(C_1m_{R_0}(|f|),\fz)$,
where $m_R(f)$ is as in \eqref{mean valu rect}.
\end{lemma}

\begin{proof}
We only need to prove that, if $\az\in(C_1 m_{R_0}(|f|),\fz)$, then
$\rlz\setminus(3R_0)\subset\rlz\setminus U^\alpha$.

For any $x:=(x_1, x_2)\notin 3R_0$, we have $|x_1-x^1_0|\ge|I_0|$ and $|x_2-x^2_0|\ge |J_0|$. Then it is easy
to show that, for any rectangle $R\ni (x_1, x_2)$ satisfying $|I|\le |I_0|$ or $|J|\le |J_0|$,
$R\cap R_0=\emptyset$.
Then
$$
\cm_sf(x)=\sup_{I\ni x_1,\,|I|\ge |I_0|}\sup_{J\ni x_2,\,|J|\ge |J_0|}
\frac1{m_\lz(I)m_\lz(J)}\int_{I}\int_J|f(y_1,y_2)|\,d\mu_\lambda(y_1,\,y_2).
$$
For any rectangle $R:=I\times J\ni (x_1, x_2)$ such that
$|I|\ge |I_0|,  |J|\ge |J_0|$ and $R\cap R_0\neq\emptyset$,
it is easy to see that $R_0\subset  3R$.
This, together with $\supp (f)\subset R_0$ implies that
\begin{eqnarray*}
\frac1{\mu_\lambda(R)}\int_R|f(y_1,y_2)|\,d\mu_\lambda(y_1, y_2)
&&\le\frac{\mu_\lambda (R_0)}{\mu_\lambda(R)}\frac1{\mu_\lambda(R_0)}\int_{R_0}|f(y_1, y_2)|\,d\mu_\lambda(y_1, y_2)\\
&&\le\frac{\mu_\lambda(3R)}{\mu_\lambda(R)}m_{R_0}(|f|)
\le C_1m_{R_0}(|f|).
\end{eqnarray*}
Thus, we have $\cm_sf(x)\le C_1 m_{R_0}(|f|)$. Moreover,
if $\az> C_1 m_{R_0}(|f|)$, then $\az>\cm_sf(x_1, x_2)$,
that is, $(x_1, x_2)\notin U^\alpha$, which completes the proof of
Lemma \ref{l7.5}.
\end{proof}

\begin{proof}[Proof of Theorem \ref{t-atm Hard q indepen}]
We have observed that $h^{1,\,\fz}(\rlz)\subset \nhoq$ for
 $q\in(1, \fz)$. Thus, we only need to establish the
converse. We do so by showing that for
any $(1,q)$-atom $a$ with $\supp(a)\subset R_0$,
$b:=\mu_\lz(R_0)\cdot a$  has an atomic decomposition
$b=\sum_{i=0}^\fz \az_ib_i$, where each $b_i$, $i\in\zz_+$,
is a $(1,\fz)$-atom and
\begin{equation*}
\sum_{i=0}^\fz|\az_i|\ls 1.
\end{equation*}
We show this by induction. In order to state the inductive
hypothesis we first introduce some necessary notation.

For each positive integer $n$, let $\nn^n$ denote the $n$-fold
Cartesian product of the natural numbers $\nn$,
$\nn^0:=\{0\}$. We write $i_n$
to represent a general element of $\nn^n$.  The inductive hypothesis we establish is the following one:

{\it There exists a collection of rectangles
$\{R_{i_l}\},  i_l\in \nn^l$
for $l\in\{1, 2, \ldots\}, $ such that, for each $n\in\nn$,
\begin{eqnarray}\label{7.10}
b&=&\sum_{l=1}^{n-1}\sum_{i_l\in \nn^l}M  C_\lz \alpha^{l+1}
\mu_\lambda(3 R_{i_l}) a_{i_l}+  \sum_{i_n\in \nn^n}h_{i_n}
=:{\rm G}_n+{\rm H}_n,
\end{eqnarray}
where $p\in(1, q)$,
$\alpha\in(1, \fz)$ is large enough which depends on $p$, $q$ and
is to be fixed later,  $C_\lz$ satisfies for any rectangle $R\subset \rlz$, $\mu_\lz(9R)\le C_\lz\mu_\lz(R)$,
and
\begin{itemize}
  \item [{\rm (I)}]
 $a_{i_l}$ is a $(1,\fz)$-atom supported in
$3R_{i_l}$,  $l\in\{1,\,2,\,\ldots,\,n-1\}$, $i_l\in \nn^l$;
  \item [{\rm (II)}] $\cup_{i_n\in \nn^n}R_{i_n}\subset\{x\in\rlz:\,\,
  \cm_{s,\,p}b(x)>\alpha^n/2\}$, where $p\in(1, q)$ and $ \cm_{s,\,p}(f):=\lf[\cm_s(|f|^p)\r]^{1/p};$
  \item [{\rm (III)}]$\{3R_{i_l}\}$ is an $M^l$-disjoint collection;
  \item [{\rm (IV)}] the function $h_{i_n}$ is supported in $R_{i_n}$
  for each $i_n\in \nn^n$;
  \item [{\rm(V)}]$\int_{\rlz} h_{i_n}(x)\,d\mu_\lz(x)=0$ for each $i_n\in \nn^n$;
  \item [{\rm (VI)}]$|h_{i_n}(x)|\le |b(x)|+2 C_\lz ^{1/p}
  \alpha^n\chi_{R_{i_n}}(x)$ for each $i_n\in \nn^n$, where $\chi_{R_{i_n}}$ is the characteristic function of $R_{i_n}$;
  \item [{\rm (VII)}]$[m_{R_{i_n}}(|h_{i_n}|^p)]^{1/p}\le 2 C_\lz ^{1/p}\alpha^n$
  for each $i_n\in \nn^n$.
\end{itemize}}

We begin with proving that
\begin{eqnarray}\label{7.11}
{\rm I}_p&:=&\frac1{\mu_\lz(R_0)}\sum_{n=1}^\fz\sum_{i_n\in \nn^n}M  C_\lz
\alpha^{n+1}\mu_\lz \lf(3R_{i_n}\r)
\ls1.
\end{eqnarray}
Indeed, from (III), (II), $b=\mu_\lz(R_0)a$ and the boundedness of $\cm_{s,\,p}$ from $L^q(\rlz)$
to $L^{q,\,\fz}(\rlz)$, we deduce that
\begin{eqnarray}\label{7.12}
\sum_{i_n\in \nn^n}\mu_\lz \lf(3 R_{i_n}\r)&\le &
C_\lz  M^n\mu_\lz \lf(\bigcup_{i_n\in \nn^n}R_{i_n}\r)\\
&\le& C_\lz  M^n\mu_\lz \lf(\{x\in\rlz:\,\, \cm_{s,\,p} b(x)>\alpha^n/2\}\r)\nonumber\\
&\ls& C_\lz  M^n 2^q \alpha^{-nq}\|b\|_{L^q(\rlz)}^q\noz\\
&\ls &  C_\lz  M^n 2^q \alpha^{-nq}\mu_\lz(R_0).\nonumber
\end{eqnarray}
This fact implies that
\begin{eqnarray*}
{\rm I}_p&\ls &
M  C_\lz \sum_{n=1}^\fz\alpha^{n+1} C_\lz  M^n 2^q \alpha^{-nq}
\approx M  C^2_\lz \alpha 2^q\sum_{n=1}^\fz(\alpha^{1-q}M)^n\ls1,
\end{eqnarray*}
if $\alpha$ is large enough such that $\alpha^{1-q}M<1$, which gives \eqref{7.11}.

By (IV), (VII), H\"older's inequality and \eqref{7.12}, we obtain
\begin{eqnarray*}
\dint_{\rlz}|{\rm H}_n(x)|\,d\mu_\lz(x)
&&\le
 \sum_{i_n\in \nn^n}\dint_{\rlz} \lf|h_{i_n}(x)\r|\,d\mu_\lz(x)\nonumber\\
&&\le
  2 C_\lz ^{1/p}\alpha^n\sum_{i_n\in \nn^n}\mu_\lz \lf(R_{i_n}\r)\nonumber\\
&&\ls
  2 C_\lz ^{1/p}\alpha^n C_\lz  M^n 2^q \alpha^{-nq}\|b\|_{L^q(\rlz)}^q\nonumber\\
&&\ls\lf(M\alpha^{1-q}\r)^n\|b\|_{L^q(\rlz)}^q.\nonumber
\end{eqnarray*}
This, together with $q>1$, shows that
${\rm G}_n$ converges to $b$ in $L^1(\mu)$.
Then the representation \eqref{7.10} holds true in $\lozd$.

Let us show that the hypothesis is valid for $n=1$. Let
$$U^\alpha:=\lf\{(x_1, x_2)\in\rlz:\,\, \cm_{s,\,p} b(x_1, x_2)>\alpha\r\}.$$
Observe that
$m_{R_0}(|b|)\le1.$
By this and Lemma \ref{l7.5}, we find that
$U^\alpha \subset 3R_0$ provided $\alpha^p>C_1$ therein.
Moreover, $U^\alpha$
is a bounded open set.  By the boundedness of $\cm_{s,\,p}$ from
$L^q(\rlz)$ to $L^{q,\,\fz}(\rlz)$, we conclude that
there exists a positive constant $C_{p,\,q}$ such that,
\begin{eqnarray*}
\mu_\lz \lf(U^\alpha\r)\le C_{p,\,q}\alpha^{-q}\lf\|b\r\|_{L^q(\rlz)}^q
&\le&C_{p,\,q}\alpha^{-q}\mu_\lz(R_0).
\end{eqnarray*}
If  $\alpha^q> C_{p,\,q}$, then $\mu_\lz(U^\alpha)<\mu_\lz(R_0)<\fz$.
We see that,  $\rlz\setminus U^\alpha$ can not be empty.
Applying Lemma \ref{l-Whitney type covering lemma} 
with ${\wz C}=3$ therein,
we obtain
a
sequence of rectangles (cubes actually) $\{R_i\}_i$
satisfying (i) through (iii) therein. Let
$\chi_i:=\chi_{R_i}$,
$$\eta_i(x):=\left\{
  \begin{array}{ll}
    \dfrac{\chi_i(x)}{\sum_k\chi_k (x)}, \quad& \hbox{if $x\in U^\alpha$;}\\[12pt]
    0, \quad & \hbox{otherwise,}
  \end{array}
\right.
$$
$$g_0 (x):=\left\{
            \begin{array}{ll}
               b (x), \quad & \hbox{if $x\notin U^\alpha $;} \\[9pt]
              \dsum_im_{R_i}(\eta_i b )\chi_i(x),
              \quad & \hbox{if $x\in U^\alpha $}
            \end{array}
          \right.
$$
and
$$h_i(x):=\eta_i(x) b (x)-m_{R_i}(\eta_i b )\chi_i(x)$$
 for all $x\in\R_+\times\R_+$.  It follows that $ b =g_0 +\sum_ih_i$.
For almost every $x\notin U^\alpha $, we see that
$$|g_0 (x)|=\lf| b (x)\r|\le \cm_{s,\,p} b (x)\le \alpha.$$
If $x\in U^\alpha $, by the H\"older inequality,
(ii) and (iii) of Lemma \ref{l-Whitney type covering lemma} and the definition of $U^\alpha $, we obtain
\begin{eqnarray}\label{7.14}
|g_0 (x)|&\le& \sum_i\dfrac1{\mu_\lz(R_i)}\dint_{R_i}
\lf|\eta_i(y) b (y)\r|\,d\mu_\lz(y)\chi_i(x)\\
&\le&\sum_i\dfrac{\mu_\lz(9R_i)}{\mu_\lz(R_i)}\lf[\dfrac1{\mu_\lz(9R_i)}
\dint_{9R_i}\lf| b (y)\r|^p\,d\mu_\lz(y)\r]^{1/p}\chi_i(x)\nonumber\\
&\le&\sum_i  C_\lz \alpha\chi_i(x)\nonumber\\
&\le& M  C_\lz \alpha.\nonumber
\end{eqnarray}
Combining these two estimates, we conclude that, for almost every $x\in\R_+\times\R_+$,
\begin{equation}\label{7.15}
|g_0 (x)|\le M  C_\lz  \alpha.
\end{equation}
We have seen that $U^\alpha \subset 3R_0 $ and that,
for $x\notin U^\alpha $, $g (x)= b (x)$.
By $\supp( b )\subset3R_0 $, we conclude
that $\supp(g )\subset 3R_0 $. Also,
\begin{equation}\label{7.16}
\supp(h_i)\subset R_i
\end{equation}
 and
\begin{equation}\label{7.17}
\int_{\rlz} h_i(x)\,d\mu_\lz(x)=0
\end{equation}
 for any $i$. Since $\{R_i\}_i$ are $M$-disjoint,
  we  have
\begin{eqnarray}\label{7.18}
\sum_i\lf\|h_i\r\|_{L^1(\rlz)}&\le&2\sum_i\lf\|\eta_i b \r\|_{L^1(\rlz)}
\le2\sum_i\dint_{R_i}\lf| b (x)\r|\,d\mu_\lz(x)\\
&\le& 2M\dint_{U^\alpha }\lf| b (x)\r|\,d\mu_\lz(x)\nonumber\\
&\le&2M\mu_\lz (R_0).\nonumber
\end{eqnarray}

Observe that $\int_{\rlz} g_0(x)\,d m_\lz(x)=0.$ Thus,
\begin{equation}\label{7.19}
a_0:=g_0 /(M C_\lz \alpha\mu_\lz(3R_0 ))
\end{equation}
is a $(1, \fz)$-atom supported in $3R_0$, and we have
\begin{eqnarray*}
b&=& M C_\lz \alpha\mu_\lz(3R_0 ) a_0
+  \sum_ih_i.
\end{eqnarray*}
This shows (I).

Now observe that
$$\bigcup_i R_i=U^\alpha =\lf\{x\in\R_+\times\R_+:\,\,
\cm_{s,\,p} b (x)>\alpha\r\}\subset\lf\{x\in\R_+\times\R_+:\,\, \cm_{s,\,p} b (x)>\alpha/2\r\}.$$
This shows (II).

Since $0\le\eta_i\le 1$, arguing as in \eqref{7.14}, we obtain
\begin{eqnarray*}
\lf|h_i(x)\r|&\le& \lf|\eta_i(x) b (x)\r|
+\lf|m_{R_i}(\eta_i b )\r|\chi_i(x)\\
&\le&\lf| b (x)\r|+\lf[m_{R_i}(| b |^p)\r]^{1/p}\chi_i(x)\\
&\le&\lf| b (x)\r|+ C_\lz ^{1/p}\alpha\chi_i(x).
\end{eqnarray*}
Thus, (VI) holds true. From this together with
the definition of $U^\alpha$ and Lemma \ref{l-Whitney type covering lemma} (iii),
we further deduce that
\begin{eqnarray*}
\lf[m_{R_i }\lf(|h_{i} |^p\r)\r]^{1/p}&\le&
\lf[m_{R_i }\lf(| b |^p\r)\r]^{1/p}+ C_\lz ^{1/p}\alpha\\
&\le& \lf[\dfrac{\mu_\lz(9R_i)}{\mu_\lz(R_i)}
m_{9R_i }\lf(| b |^p\r)\r]^{1/p}+ C_\lz ^{1/p}\alpha\\
&\le& 2 C_\lz ^{1/p}\alpha,
\end{eqnarray*}
which implies (VII). Moreover, (III) is a consequence of Lemma \ref{l-Whitney type covering lemma}(ii),
and (IV) holds true by \eqref{7.16} and
(V) holds true by \eqref{7.17}. This shows that the induction holds true for $n=1$.

We now assume that the hypothesis holds true for $n$
and show that it is also valid for $n+1$.
Let
$$U^\az_{i_n} :=\lf\{x\in\rlz:\,\, \cm_{s,\,p}h_{i_n}(x)>\alpha^{n+1}\r\}.$$
By (IV) for $n$, we have
$\supp(h_{i_n})\subset R_{i_n}.$
Moreover, it follows, from (VII) for $n$,
provided $\alpha^p>2^pC_1C_\lz$, that
$$C_1m_{R_{i_n}}\lf(\lf|h_{i_n}\r|^p\r)
\le C_1C_\lz (2\alpha^n)^p<\alpha^{(n+1)p}.$$
By Lemma \ref{l7.5}, we see that
\begin{equation}\label{7.20}
U^\az_{i_n} =\lf\{x\in\rlz:\,\,
\cm_s\lf(\lf|h_{i_n}\r|^p\r)(x)>\alpha^{(n+1)p}\r\}\subset 3R_{i_n}.
\end{equation}
Let rectangles $\{R_{i_n,\,k}\}_k$
be a Whitney covering of $U^\az_{i_n} $.
From (i) and (ii) of Lemma \ref{l-Whitney type covering lemma} and
\eqref{7.20}, it follows that
$$\bigcup_k3R_{i_n,\,k}=U^\az_{i_n} \subset {3R_{i_n}}$$
and $\{{3R_{i_n,\,k}}\}_k$ is $M$-disjoint. Since, from (III) for $n$, we know
$\{{3R_{i_n}}\}_{i_n}$ are $M^n$-disjoint,
it follows that the totality of rectangles (cubes) in the family
$\{3R_{i_n,\,k}\}_{k,\,i_n}$ are $M^{n+1}$-disjoint.
This establishes (III) for $n+1$.

We now put
$$g_{i_n} (x):=\left\{
            \begin{array}{ll}
              h_{i_n}(x), \quad & \hbox{if $x\notin U^\az_{i_n} $;} \\
              \dsum_km_{R_{i_n,\,k}}(\eta^{i_n}_kh_{i_n})
              \chi_{R_{i_n,\,k}}(x), \quad & \hbox{if $x\in U^\az_{i_n} $}
            \end{array}
          \right.
$$
and
$$h_{i_n,\,k}:=\eta^{i_n}_kh_{i_n}-m_{R_{i_n,\,k}}
(\eta^{i_n}_kh_{i_n})\chi_{R_{i_n,\,k}},$$
where
$$\eta^{i_n}_k(x):=\chi_{R_{i_n,\,k}}(x)\Big/\sum_k\chi_{R_{i_n,\,k}}(x)$$
for $x\in U^\az_{i_n} $, and is $0$ if $x\notin U^\az_{i_n} $. If $x\in U^\az_{i_n} $, then
\begin{eqnarray*}
\lf|g_{i_n} (x)\r|&\le &\sum_{k}\lf|m_{R_{i_n,\,k}}
(\eta^{i_n}_kh_{i_n})\chi_{R_{i_n,\,k}}(x)\r|\\
&\le&\sum_{k}\dfrac{\mu_\lz(9R_{i_n,\,k} )}{\mu_\lz(R_{i_n,\,k})}
\dfrac1{\mu_\lz(9R_{i_n,\,k} )}
\dint_{9R_{i_n,\,k} }\lf|h_{i_n}(y)\r|\,d\mu_\lz(y)
\chi_{R_{i_n,\,k}}(x)\\
&\le& M C_\lz \alpha^{n+1},
\end{eqnarray*}
while if $x\notin U^\az_{i_n} $, then
$$\lf|g_{i_n} (x)\r|=\lf|h_{i_n}(x)\r|\le \cm_{s,\,p}h_{i_n}(x)\le \alpha^{n+1}.$$
In any case, we have
$$\lf\|g_{i_n} \r\|_{L^\fz(\mu)}\le M C_\lz \alpha^{n+1}.$$
Since the support of $h_{i_n}$ is within
$R_{i_n}\subset 3R_{i_n}$ and $U^\az_{i_n} \subset 3R_{i_n}$,
it follows that the support of $g_{i_n} $ is included in  $3R_{i_n}$. Moreover,
$\int_{\rlz} h_{i_n,\,k}(x)\,d\mu_\lz(x)=0$
(which shows that property (V) is valid for $n+1$).
By an argument used in the estimate for \eqref{7.18},
it is easy to see that
$$\dsum_k \lf\|h_{i_n,\,k}\r\|_{L^1(\rlz)}\le 2M\lf\|h_{i_n}\r\|_{L^1(\rlz)}.$$
It then follows from this that
$$h_{i_n}=g_{i_n} +\dsum_k h_{i_n,\,k}$$
is valid in $L^1(\mu)$ and $\int_{\rlz} g_{i_n} (x)\,d\mu_\lz(x)=0$.

Let
$$a_{i_n} :=g_{i_n} /\lf\{M  C_\lz \alpha^{n+1}
\lf[\mu_\lz \lf(3R_{i_n}\r)\r]\r\}.$$
Then $a_{i_n}$
is a $(1, \fz)$-atom supported in the rectangle $3R_{i_n}$.
From this, we deduce that \eqref{7.10} holds true for $n+1$ and so does  (I).
Property (IV) is trivially true.
Moreover, by the definition of $h_{i_n,\,k}$, (VI) for $n$
and Lemma \ref{l-Whitney type covering lemma}(iii), we conclude that
\begin{eqnarray*}
\lf|h_{i_n,\,k}(x)\r|&\le&\lf\{\lf|h_{i_n}(x)\r|
+\lf[ C_\lz \dfrac1{\mu_\lz(9R_{i_n,\,k} )}
\dint_{9R_{i_n,\,k} }\lf|h_{i_n}(x)\r|^p\,d\mu_\lz(x)\r]^{1/p}\r\}
\chi_{R_{i_n,\,k}}(x)\\
&\le&\lf\{\lf| b (x)\r|+2 C_\lz ^{1/p}\alpha^n+ C_\lz ^{1/p}
\alpha^{n+1}\r\}\chi_{R_{i_n,\,k}}(x)\\
&\le&\lf\{\lf| b (x)\r|+2 C_\lz ^{1/p}\alpha^{n+1}\r\}\chi_{R_{i_n,\,k}}(x)
\end{eqnarray*}
if $\alpha>2$. This establishes (VI) for $n+1$.

On the other hand, by the definitions of $h_{i_n,\,k}$
and $U^\az_{i_n} $, we have
\begin{eqnarray*}
\lf[m_{R_{i_n,\,k}}\lf(\lf|h_{i_n,\,k}\r|^p\r)\r]^{1/p}&\le& 2
\lf[m_{R_{i_n,\,k}}\lf(\lf|h_{i_n}\r|^p\r)\r]^{1/p}\\
&\le& 2\lf[ C_\lz m_{9R_{i_n,\,k}}
\lf(\lf|h_{i_n}\r|^p\r)\r]^{1/p}\\
&\le&2 C_\lz ^{1/p}\alpha^{n+1},
\end{eqnarray*}
which shows (VII).

Finally, from (VI) for $n$, we deduce that
$$\cm_{s,\,p}(h_{i_n})(x)\le \cm_{s,\,p}( b )(x)+2 C_\lz ^{1/p}\alpha^{n}.$$
Thus, if $x\in U^\az_{i_n} $, then
$$\alpha^{n+1}<\cm_{s,\,p}(h_{i_n})(x)\le \cm_{s,\,p}( b )(x)
+2 C_\lz ^{1/p}\alpha^{n}.$$
It follows that, if $\alpha> 4 C_\lz ^{1/p}$, then
$\alpha^{n+1}/2<\cm_{s,\,p}( b )(x).$
Thus,
$$
\bigcup_{i_n,\,k}R_{i_n,\,k}=\bigcup_{i_n}\bigcup_kR_{i_n,\,k}
\subset \bigcup_{i_n}U^\az_{i_n}
\subset\lf\{x\in\rlz:\,\, \cm_{s,\,p}(b)(x)>\alpha^{n+1}/2\r\}
$$
and (II) is valid for $n+1$.
This finishes the proof of Theorem \ref{t-atm Hard q indepen}.
\end{proof}

Based on Theorem \ref{t-atm Hard q indepen}, we now denote by
$h^1(\rlz)$ the little Hardy space, and  we have the following result on the duality of $h^1(\rlz)$ with $\lbmo(\rlz)$.
\begin{thm} \label{thm dual}
The predual of bmo$(\rlz)$ is $h^{1}(\rlz)$.
\end{thm}

\begin{proof}

The duality of $h^{1,2}(\rlz)$ with bmo$(\rlz)$ follows from a standard argument, see for example \cite{CW2} (see also \cite[Section II, Chapter 3]{J}). Hence, by Theorem \ref{t-atm Hard q indepen}, the  predual of bmo$(\rlz)$ is $h^{1,\infty}(\rlz)$.
\end{proof}

Our main result of this section is the following.

\begin{thm}\label{thm: weak factotrization}
For every $f\in h^{1}(\rlz)$, there exist sequences $\{\alpha_j^k\}_{j}\in\ell^1$ and functions
$g_j^k,h^k_j\in L^\infty(\rlz)$ with compact support, such that
\begin{align}\label{represent of H1}
f=\sum_{k=1}^\fz\sum_{j=1}^\fz\alpha^k_j\, \Pi\lf(g^k_j,h^k_j\r)
\end{align}
in the sense of $h^1(\rlz)$, where $\Pi(g,h )$ is the bilinear form defined as
\begin{align}\label{bilinear form}
\Pi(g,h) :=  g\cdot \rizo\rizt (h) - h\cdot \wrizo\wrizt (g),
\end{align}
where $\wrizo$ and $\wrizt$ are the adjoints of $\rizo$ and $\rizt$, respectively.

Moreover, we have that 
\begin{eqnarray*}
\|f\|_{h^1(\rlz)}\approx\inf\Big\{\sum_{k=1}^\fz\sum_{j=1}^\fz\lf|\az^k_j\,\r|\lf\|g^k_j\r\|_{L^2(\rlz)}\lf\|h^k_j\r\|_{L^2(\rlz)}\Big\},
\end{eqnarray*}
where the infimum is taken over all representations of $f$ in the form \eqref{represent of H1} and the implicit constants
are independent of $f$.
\end{thm}

To prove Theorem \ref{thm: weak factotrization}, we study the property of the bilinear form $\Pi(f,g)$ as defined in \eqref{bilinear form}, which connects to the commutator $\left[  b, \rizo\rizt \right]$.
%
%
%

\begin{prop}\label{t-H1 estimate of pi}
For every $g,h \in L^\infty(\rlz)$ with compact support, the bilinear form $\Pi(g,h)$
is in $h^1(\rlz)$
with the norm satisfying
\begin{align}\label{boundedness of bilinear form}
\| \Pi(g,h) \|_{h^1(\rlz)} \le C \|g\|_{L^2(\rlz)} \| h \|_{L^2(\rlz)}.
\end{align}

\end{prop}
\begin{proof}
First, it is clear that for every $g,h \in L^\infty(\rlz)$ with compact support, the bilinear form $\Pi(f,g)$
is in $L^1(\rlz)$ with compact support and satisfies
$$ \int_{\R_+\times\R_+} \Pi(g,h) (x_1,x_2)\dmz(x_1)\dmz(x_2)=0.$$

Moreover, for $b\in {\rm bmo}(\rlz)$ and for every $g,h \in L^\infty(\rlz)$ with compact support, we have
\begin{align}\label{PiNorm}
\left\vert \left \langle b, \Pi(g,h)\right \rangle_{L^2(\rlz)} \right\vert
& =\left\vert \left \langle b, g \rizo\rizt h - h \wrizo\wrizt g \right \rangle_{L^2(\rlz)} \right\vert\\
& = \left\vert \left \langle  \left[  b, \rizo\rizt \right] f ,g \right \rangle_{L^2(\rlz)}\right\vert\nonumber\\
& \lesssim \|b \|_{{\rm bmo}(\rlz)} \|f \|_{L^2(\rlz)} \|g \|_{L^2(\rlz)},\nonumber
\end{align}

This, together with the duality result as in Theorem \ref{thm dual}, implies that for every $g,h \in L^\infty(\rlz)$ with compact supports, the bilinear form $\Pi(f,g)$ is in
$h^1(\rlz)$. Moreover, the $h^1(\rlz)$ norm of  $\Pi(f,g)$ satisfies \eqref{boundedness of bilinear form}. In fact, we point out that from the fundamental fact as in \cite[Exercise 1.4.12 (b)]{Gra}, we have
$$ \|\Pi(g,h)\|_{ h^1(\rlz)} \approx \sup_{b:\ \|b\|_{ {\rm bmo}(\rlz)}\leq1  }
\big| \langle  b, \Pi(g,h) \rangle_{\ltp} \big|, $$
which, together with \eqref{PiNorm}, immediately implies that
\eqref{boundedness of bilinear form} holds.
\end{proof}

Next, we provide the following approximation to each $h^{1,\infty}(\rlz)$ atom via the bilinear form defined as in \eqref{bilinear form}.
\begin{thm} \label{thm:recfac} Let $\epsilon$ be an arbitrary positive number.  Let $a(x_1,x_2)$ be
an $\fz$-atom as defined in Definition \ref{def atom}.
Then there exist two functions $f,g \in L^\infty(\rlz)$ with compact supports and a constant $C(\epsilon)$ depending only on $\epsilon$ such that
\[ \| a  - \Pi(f,g) \|_{h^1(\rlz)} < \epsilon, \]
where $\| f \|_{\ltp} \|g \|_{\ltp} \le C(\epsilon)$.
\end{thm}

To prove Theorem \ref{thm:recfac}, we first provide a technical lemma as follows.

\begin{lem} \label{lem:H1Est}
Let $R:=I(x_{0,1},r_1)\times I(x_{0,2},r_2)$ and $\wz R:=I(y_{0,1},r_1)\times I(y_{0,2},r_2)$ be two rectangles in $\R_+\times\R_+$ with $r_1\leq \min\{x_{0,1}, y_{0,1}\}$ and $r_2\leq \min\{x_{0,2}, y_{0,2}\}$. Moreover, assume that
$ |x_{0,1}-y_{0,1}|\geq 4r_1 $ and $ |x_{0,2}-y_{0,2}|\geq 4r_2 $.

Let $f: \mathbb{R}^2 \rightarrow \mathbb{C}$ with
 $\textnormal{supp}\, f \subseteq R\cup \widetilde{R}.$
  Further, assume that
  $$ |f(x_1,x_2)| \ls  \wz C_1\chi_R(x_1,x_2)+\wz C_2\chi_{\widetilde{R}}(x_1,x_2) $$
  and that $f$ has a mean value zero property:
\begin{align}\label{mean zero var fir}
\int_{\R_+\times\R_+} f(x_1,x_2) \,\dmz(x_1)\dmz(x_2) = 0.
\end{align}
Then there exists a positive constant $C$ independent of  $x_{0,1}$, $x_{0,2}$,  $y_{0,1}$, $y_{0,2}$, $r_1$, $r_2$, $\wz C_1$ and $\wz C_2$ such that
$$\| f \|_{h^1(\rlz)} \leq C\bigg( \log_2 {|x_{0,1}-y_{0,1}| \over r_1}+\log_2 {|x_{0,2}-y_{0,2}| \over r_2} \bigg)\Big(\wz C_1\mu_\lz(R)+\wz C_2\mu_\lz(\widetilde{R}) \Big). $$
\end{lem}

\begin{proof}
Suppose $f$ satisfies the conditions as stated above. We will show that
$f$ has an atomic decomposition as the form in Definition \ref{def little h1}.  To see this, we first define two functions $f_1(x_1,x_2) $ and $f_2(x_1,x_2)$ by
\begin{align*}
&f_1(x_1,x_2)=f(x_1,x_2), (x_1,x_2)\in R; \quad f_1(x_1,x_2)=0, (x_1,x_2)\in \R^2\setminus R, \quad{\rm and}\\
& f_2(x_1,x_2)=f(x_1,x_2), (x_1,x_2)\in \widetilde{R}; \quad f_2(x_1,x_2)=0, (x_1,x_2)\in \R^2\setminus \widetilde{R}.
\end{align*}
Then we have that $f=f_1+f_2$ and that
$$ |f_1(x_1,x_2)| \ls  \wz C_1 \chi_R(x_1,x_2)\quad {\rm and}\quad |f_2(x_1,x_2)| \ls  \wz C_2  \chi_{\widetilde{R}}(x_1,x_2). $$

Define
\begin{align*}
g_1^1(x_1,x_2)&:=\frac{\chi_{2R}(x_1,x_2)}{\mu_\lz(2R)}\iint_{R}f_1(y_1,y_2)\dmz(y_1)\dmz(y_2),\\
f_1^1(x_1,x_2)&:= f_1(x_1,x_2)- g_1^1(x_1,x_2),\\
\alpha_1^1&:=\|f_1^1\|_{L^\infty(\rlz)} \mu_\lz(2R).
\end{align*}
Then we claim that $a_1^1:= (\alpha_1^1)^{-1} f_1^1$  is a rectangle atom as in Definition \ref{def atom}.
First, it is direct that $a_1^1$ is supported in $2R$. Moreover, we have that
\begin{align*}
  &\int_{\R_+\times\R_+} a_1^1(x_1,x_2) \,\dmz(x_1)\dmz(x_2)\\
  &\quad =(\alpha_1^1)^{-1}\int_{\R_+\times\R_+}  \left(f_1(x_1,x_2)- g_1^1(x_1,x_2) \right)\dmz(x_1)\dmz(x_2)\\
  &\quad= (\alpha_1^1)^{-1}\bigg(\int_{\R_+\times\R_+}  f_1(x_1,x_2)\,\dmz(x_1)\dmz(x_2)- \int_{\R_+\times\R_+}  f_1(x_1,x_2)\dmz(x_1)\dmz(x_2)\bigg)\\
  &\quad=0
\end{align*}
and that
\begin{align*}
\|a_1^1\|_{L^\infty(\rlz)} \leq |(\alpha_1^1)^{-1}| \|f_1^1\|_{L^\infty(\rlz)} = \frac1{\mu_\lz(2R)}.
\end{align*}
Thus, $a_1^1$  is an $\fz$-atom as in Definition \ref{def atom}. Moreover, we have
\begin{align*}
\alpha_1^1=\|f_1^1\|_{L^\infty(\rlz)}  \mu_\lz(2R) \le  \|f_1\|_{L^\infty(\rlz)} \mu_\lz(2R) + \|g_1^1\|_{L^\infty(\rlz)} \mu_\lz(2R) \ls \wz C_1\mu_\lz(R),
\end{align*}
where the implicit constant depends only on $\lz$. We now have
\begin{align*}
f_1(x_1,x_2)= f_1^1(x_1,x_2)+ g_1^1(x_1,x_2)= \alpha_1^1 a_1^1+ g_1^1(x_1,x_2).
\end{align*}
For $g_1^1(x_1,x_2)$, we further write it as
\begin{align*}
g_1^1(x_1,x_2)= g_1^1(x_1,x_2)- g_1^2(x_1,x_2)+ g_1^2(x_1,x_2)=:f_1^2(x_1,x_2)+g_1^2(x_1,x_2)
\end{align*}
with
$$g_1^2(x_1,x_2):=\frac{\chi_{4R}(x_1,x_2)}{\mu_\lz(4R)}\iint_{R}f_1(y_1,y_2)\dmz(y_1)\dmz(y_2). $$
Again, we define
\begin{align*}
\alpha_1^2&:=\|f_1^2\|_{L^\infty(\rlz)}  \mu_\lz(4R) \quad{\rm and}\quad a_1^2:= (\alpha_1^2)^{-1} f_1^2,
\end{align*}
and following similar estimates as for $a_1^1$, we see that
$a_1^2$  is an $\fz$-atom as in Definition \ref{def atom} with
$$ \|a_1^2\|_{L^\infty(\rlz)} \leq \frac{1}{\mu_\lz(4R)}\quad \textnormal{ and } \quad  \alpha_1^2 \lesssim \wz C_1\mu_\lz(R), $$
where the implicit constant depends only on $\lz$.

Then we have
\begin{align*}
f_1(x_1,x_2)=\sum_{i=1}^2 \alpha_1^i a_1^i+ g_1^2(x_1,x_2).
\end{align*}
Continuing in this fashion we see that for $i \in \{1, 2, . . . , i_0\}$,
\begin{align*}
f_1(x_1,x_2)=\sum_{i=1}^{i_0} \alpha_1^i a_1^i+ g_1^{i_0}(x_1,x_2),
\end{align*}
where for $i \in \{2, . . . , i_0\}$,
\begin{align*}
g_1^{i}(x_1,x_2)&:=\frac{\chi_{2^{i}R}(x_1,x_2)}{\mu_\lz(2^i R)}\iint_{R}f_1(y_1,y_2)\dmz(y_1)\dmz(y_2),\\
f_1^{i}(x_1,x_2)&:= g_1^{i-1}(x_1,x_2)-g_1^{i}(x_1,x_2),\\
\alpha_1^i&:=\|f_1^i\|_{L^\infty(\rlz)} \mu_\lz(2^iR) \quad{\rm and}\\
a_1^i&:= (\alpha_1^i)^{-1} f_1^i.
\end{align*}
Here we choose $i_0$ to be the smallest positive integer such that
$$ i_0\geq  \log_2 {|x_{0,1}-y_{0,1}| \over r_1}+\log_2 {|x_{0,2}-y_{0,2}| \over r_2} .$$

Moreover, for $i \in \{1, 2, . . . , i_0\}$, we have
\begin{align*}
\alpha_1^i&\leq \|f_1^i\|_{L^\infty(\rlz)} \mu_\lz(2^iR) \leq  \big(\|g_1^{i-1}\|_{L^\infty(\rlz)} +\|g_1^{i}\|_{L^\infty(\rlz)}\big)\ \mu_\lz(2^iR)\\
&\leq \mu_\lz(2^iR)  \bigg(\frac{1}{\mu_\lz(2^{i-1} R)}\iint_{R}|f_1(y_1,y_2)|\dmz(y_1)\dmz(y_2) \\
&\hskip3cm+{1\over \mu_\lz(2^iR)}\iint_{R}|f_1(y_1,y_2)|\dmz(y_1)\dmz(y_2)\bigg)\\
&\ls \mu_\lz(2^iR) \frac{1}{\mu_\lz(2^{i-1} R)} \|f_1\|_{L^\infty(\rlz)} \mu_\lz(R) \\
&\ls \wz C_1 \mu_\lz(R),
\end{align*}
where the implicit constant depends only on $\lz$.

Following the same steps, we also obtain that  for $i \in \{1, 2, . . . , i_0\}$,
\begin{align*}
f_2(x_1,x_2)=\sum_{i=1}^{i_0} \alpha_2^i a_2^i+ g_2^{i_0}(x_1,x_2),
\end{align*}
where for $i \in \{2, . . . , i_0\}$,
\begin{align*}
g_2^{i}(x_1,x_2)&:={\chi_{2^{i}\wz R}(x_1,x_2)\over \mu_\lz(2^i \wz R)}\iint_{\widetilde R}f_2(y_1,y_2)\dmz(y_1)\dmz(y_2),\\
f_2^{i}(x_1,x_2)&:= g_2^{i-1}(x_1,x_2)-g_2^{i}(x_1,x_2),\\
\alpha_2^i&:=\|f_2^i\|_{L^\infty(\rlz)} \mu_\lz(2^i\wz R) \quad{\rm and}\\
a_2^i&:= (\alpha_2^i)^{-1} f_2^i.
\end{align*}
Similarly, for $i \in \{1, 2, . . . , i_0\}$, we have
\begin{align*}
\alpha_2^i\ls  \wz C_2 \mu_\lz(\wz R).
\end{align*}

Combining the decompositions above, we obtain that
\begin{align*}
f(x_1,x_2)=\sum_{j=1}^2\sum_{i=1}^{i_0} \alpha_j^i a_j^i+ g_j^{i_0}(x_1,x_2).
\end{align*}
We now consider the tail $g_1^{i_0}(x_1,x_2) + g_2^{i_0}(x_1,x_2)$. To handle that, consider
the rectangle $\overline R$  defined as
$$\overline R:=  I\Big( {x_{0,1}+y_{0,1}\over2}, (2^{i_0}+1)r_1\Big) \times  I\Big( {x_{0,2}+y_{0,2}\over2}, (2^{i_0}+1)r_2\Big). $$
 Then, it is clear that $R\cup \widetilde{R} \subset \overline R$,
and that $2^{i_0}R, 2^{i_0}\widetilde{R} \subset \overline R$.
Thus, we get that
\begin{align*}
{\chi_{\overline{R}}(x_1,x_2)\over \mu_\lz(\overline{R})}\iint_{\overline{R}}f_1(y_1,y_2)\dmz(y_1)\dmz(y_2)+
\ {\chi_{\overline{R}}(x_1,x_2)\over \mu_\lz(\overline{R})}\iint_{\overline{R}}f_2(y_1,y_2)\dmz(y_1)\dmz(y_2)=0.
\end{align*}
Hence, we write
\begin{align*}
g_1^{i_0}(x_1,x_2) + g_2^{i_0}(x_1,x_2)&=\bigg(g_1^{i_0}(x_1,x_2) -
{\chi_{\overline{R}}(x_1,x_2)\over \mu_\lz(\overline{R})}\iint_{\overline{R}}f_1(y_1,y_2)\dmz(y_1)\dmz(y_2)\bigg) \\
&\quad+ \bigg( g_2^{i_0}(x_1,x_2) -
 {\chi_{\overline{R}}(x_1,x_2)\over \mu_\lz(\overline{R})}\iint_{\overline{R}}f_2(y_1,y_2)\dmz(y_1)\dmz(y_2)\bigg)\\
 &=: f_1^{i_0+1}+f_2^{i_0+1}.
\end{align*}
We now define
\begin{align*}
\alpha_1^{i_0+1}&:=\|f_1^{i_0+1}\|_{L^\infty(\rlz)} \mu_\lz(2^{i_0+1}R), \quad \alpha_2^{i_0+1}:=\|f_2^{i_0+1}\|_{L^\infty(\rlz)} \mu_\lz(2^{i_0+1}\wz R) \\
a_1^{i_0+1}&:= (\alpha_1^{i_0+1})^{-1} f_1^{i_0+1}   \quad{\rm and}\quad a_2^{i_0+1}:= (\alpha_2^{i_0+1})^{-1} f_2^{i_0+1}.
\end{align*}
Again we can verify that  $a_1^{i_0+1}$ is an $\fz$-atom as in Definition \ref{def atom} with
$$\|a_1^{i_0+1}\|_{L^\infty(\rlz)} = {1\over  \mu_\lz(2^{i_0+1}R) }.$$
Moreover, we also have
$$\alpha_1^{i_0+1} \ls \wz C_1  \mu_\lz(R).$$
Similarly, $a_2^{i_0+1}$ is an $\fz$-atom as in Definition \ref{def atom} with
$$\|a_2^{i_0+1}\|_{L^\infty(\rlz)} = {1\over  \mu_\lz(2^{i_0+1}\wz R) },$$
and we also have
$$\alpha_2^{i_0+1} \ls \wz C_1  \mu_\lz(\wz R).$$

Thus, we obtain that
\begin{align*}
f(x_1,x_2)=\sum_{j=1}^2\sum_{i=1}^{i_0+1} \alpha_j^i a_j^i,
\end{align*}
which implies that $f\in h^1(\rlz)$ and
\begin{align*}
\|f\|_{h^1(\rlz)} &\leq \sum_{j=1}^2\sum_{i=1}^{i_0+1}  \alpha_j^i
\\
&\leq C\bigg( \log_2 {|x_{0,1}-y_{0,1}| \over r_1}+\log_2 {|x_{0,2}-y_{0,2}| \over r_2} \bigg)\Big(\wz C_1\mu_\lz(R)+\wz C_2\mu_\lz(\widetilde{R}) \Big).
\end{align*}
Therefore, we finish the proof of Lemma \ref{lem:H1Est}.
\end{proof}

\bigskip
\begin{proof}[\bf Proof of Theorem \ref{thm:recfac}]  Suppose $a$ is an atom of $h^1(\rlz)$ supported in a rectangle
$$R:=I(x_{0,1}, r_1)\times I(x_{0,2}, r_2),$$ as in Definition \ref{def atom}. Observe that if $r_1>x_{0,1}$, then
$I(x_{0,1}, r_1)=(x_{0,1}-r_1,x_{0,1}+r_1)\cap \mathbb{R}_+=I(\frac{x_{0,1}+r_1}2, \frac{x_{0,1}+r_1}2)$. Therefore, without loss of generality, we may assume
that $r_1\le x_{0,1}$, and similarly assume that $r_2\le x_{0,2}$. Let $K_2$ and ${K}_3$ be the constants appeared in  (ii) and (iii) of Proposition \ref{p-estimate of riesz kernel} respectively, and
$K_0>\max\{{1\over K_2},\,\frac1{{K}_3}\}+1$ large enough. For any $\ez>0$, let $\wz M$ be a positive constant large enough such that $\wz M\ge100K_0$
and $\frac{\log_2 \wz M}{\wz M}<\ez$.


We now consider the following four cases.

\smallskip

Case (a): $x_{0,1}\le 2\wz Mr_1$,  $x_{0,2}\le 2\wz Mr_2$.

\smallskip
In this case, let $y_{0,1}:=x_{0,1}+2\wz MK_0r_1$ and $y_{0,2}:=x_{0,2}+2\wz MK_0r_2$ and
$$\wz R:=I(y_{0,1}, r_1)\times I(y_{0,2}, r_2).$$
Then for $i=1,2$,
$$(1+K_0)x_{0,i}\le y_{0,i}\le (1+2\wz MK_0)x_{0,i}.$$
Define
\begin{align}\label{g}
 g(x_1,x_2):=\chi_{\wz R}(x_1,x_2)
\end{align}
and
\begin{align}\label{h}
h(x_1,x_2):=-\frac{a(x_1,x_2)}{\wrizo \wrizt (g)(x_{0,1},x_{0,2})}.
\end{align}

We first point out that by the fact that $y_i/x_{0,i}>K_2^{-1}$ for any $y_i\in I(y_{0,i},r_i)$, $i=1,2$, and Proposition \ref{p-estimate of riesz kernel} (ii), we see that
\begin{align}\label{lower bound of riesz}
&\lf| \wrizo \wrizt (g)(x_{0,1},x_{0,2}) \r|\\
&\quad=\Big|\int_{y_{0,1}-r_1}^{y_{0,1}+r_1}  \wrizo(y_1, x_{0,1}) \dmz(y_1)\int_{y_{0,2}-r_2}^{y_{0,2}+r_2}  \wrizo(y_2, x_{0,2}) \dmz(y_2)\Big|\noz\\
&\quad\gs\int_{y_{0,1}-r_1}^{y_{0,1}+r_1}\frac1{y_1}dy_1\int_{y_{0,2}-r_2}^{y_{0,2}+r_2}\frac1{y_2}dy_2\sim\frac {r_1}{y_{0,1}} \frac {r_2}{y_{0,2}} \sim\frac{1}{\wz M^2}.\noz
\end{align}

Then from the definitions of $g$ and $h$ above, we have
\[ \|g \|_{\ltp} = \mu_\lz(\wz R)^{\frac{1}{2}} \]
and
\[ \| h \|_{\ltp} = \frac{1}{\big|\wrizo \wrizt (g)(x_{0,1},x_{0,2})\big|} \| a \|_{\ltp} \le \frac{\mu_\lz(R)^{-\frac{1}{2}}}{\big|\wrizo \wrizt (g)(x_{0,1},x_{0,2})\big|} .\]
Thus, from \eqref{lower bound of riesz},  we have that
\begin{align*}
 \|g \|_{\ltp}  \| h \|_{\ltp} &\ls \wz M^2  \mu_\lz(\wz R)^{\frac{1}{2}} \mu_\lz(R)^{-\frac{1}{2}}
 \ls \wz M^2 \bigg({  y_{0,1}^{2\lz}\, r_1\ y_{0,2}^{2\lz}\, r_2 \over   x_{0,1}^{2\lz}\, r_1\ x_{0,2}^{2\lz}\, r_2  } \bigg)^{1\over2} \ls \wz M^{2+2\lz}.
\end{align*}


Now, write
\begin{align*}
 &a(x_1,x_2) - \Pi(g,h)(x_1,x_2)\\
 & \quad= \left( a(x_1,x_2) + h(x_1,x_2) \wrizo \wrizt (g)(x_1,x_2) \right)  - g(x_1,x_2) \rizo\rizt (h)(x_1,x_2)\\
 & \quad=:w_1(x_1,x_2) +w_2(x_1,x_2).
 \end{align*}

Moreover, we define
$$D_1:=\frac{ m_\lz(I(y_{0,1},r_1))}
{ m_\lz(I(x_{0,1},r_1)) m_\lz(I(x_{0,1},|y_{0,1}-x_{0,1}|)) m_\lz(I(x_{0,2},|y_{0,2}-x_{0,2}|)) }$$
and
$$D_2:=\frac{ 1}
{  m_\lz(I(x_{0,1},|y_{0,1}-x_{0,1}|)) m_\lz(I(x_{0,2},|y_{0,2}-x_{0,2}|)) }.$$

First, consider $w_1.$ Observe that $\text{supp } w_1 \subseteq R$ and
\[ |w_1(x_1,x_2) |  = |a(x_1,x_2)| \frac{\left| \wrizo \wrizt (g)(x_{0,1},x_{0,2}) - \wrizo \wrizt (g)(x_{1},x_{2}) \right|}{\big|\wrizo \wrizt (g)(x_{0,1},x_{0,2})\big|}.\]
Then as $(x_1,x_2) \in R$,  we can estimate
\begin{align*}
 &\left| \wrizo \wrizt (g)(x_{0,1},x_{0,2}) - \wrizo \wrizt (g)(x_{1},x_{2}) \right|\\
  &\quad= \left|
 \int_{\widetilde{R}} \Big[\wrizo(x_{0,1},y_{1})\wrizt(x_{0,2},y_{2}) - \wrizo(x_{1},y_{1})\wrizt(x_{2},y_{2}) \Big]\dmz(y_1)\dmz(y_2) \right| \\
&\quad \ls   \int_{\widetilde{R}} \bigg[\frac{|x_1-x_{0,1}|}
{|y_1-x_{0,1}| m_\lz(I(x_{0,1},|y_1-x_{0,1}|)) m_\lz(I(x_{0,2},|y_2-x_{0,2}|)) }  \\
&\quad\quad\quad\quad+ \frac{|x_2-x_{0,2}|}
{|y_2-x_{0,2}| m_\lz(I(x_{0,2},|y_2-x_{0,2}|)) m_\lz(I(x_{1},|y_1-x_{1}|)) } \bigg]\dmz(y_1)\dmz(y_2)   \\
&\quad \ls   \mu_\lz(\widetilde{R}) \bigg[\frac{r_1}
{|y_{0,1}-x_{0,1}| m_\lz(I(x_{0,1},|y_{0,1}-x_{0,1}|)) m_\lz(I(x_{0,2},|y_{0,2}-x_{0,2}|)) }  \\
&\quad\quad\quad\quad\quad\quad+ \frac{r_2}
{|y_{0,2}-x_{0,2}| m_\lz(I(x_{0,2},|y_{0,2}-x_{0,2}|)) m_\lz(I(x_{1},|y_{0,1}-x_{1}|)) } \bigg].
\end{align*}
Combining the above estimates, \eqref{lower bound of riesz}, and the definition of $w_1$ immediately gives:
\begin{align*}
|w_1(x_1,x_2)|&\ls \wz M^2  \|a\|_{L^\infty(R)} \mu_\lz(\widetilde{R})\bigg[\frac{r_1\ m_\lz(I(y_{0,1},r_1))}
{|y_{0,1}-x_{0,1}| m_\lz(I(x_{0,1},|y_{0,1}-x_{0,1}|)) m_\lz(I(x_{0,2},|y_{0,2}-x_{0,2}|)) }  \\
&\quad\quad+ \frac{r_2\ m_\lz(I(y_{0,2},r_2))}
{|y_{0,2}-x_{0,2}| m_\lz(I(x_{0,2},|y_{0,2}-x_{0,2}|)) m_\lz(I(x_{1},|y_{0,1}-x_{1}|)) } \bigg]\chi_{R}(x_1,x_2)\\
&\ls \bigg[\frac{ m_\lz(I(y_{0,1},r_1))}
{ m_\lz(I(x_{0,1},r_1)) m_\lz(I(x_{0,1},|y_{0,1}-x_{0,1}|)) m_\lz(I(x_{0,2},|y_{0,2}-x_{0,2}|)) }  \\
&\quad\quad+ \frac{ m_\lz(I(y_{0,2},r_2))}
{ m_\lz(I(x_{0,2},r_2)) m_\lz(I(x_{0,2},|y_{0,2}-x_{0,2}|)) m_\lz(I(x_{1},|y_{0,1}-x_{1}|)) } \bigg]\chi_{R}(x_1,x_2)\\
&\ls D_1\chi_{R}(x_1,x_2).
\end{align*}

Now, consider $w_2(x_1,x_2)$.
Note that
\[ w_2(x_1,x_2) = \frac{1}{ \wrizo \wrizt (g)(x_{0,1},x_{0,2})} \chi_{\widetilde{R}} (x_1,x_2) \rizo\rizt (a)(x_1,x_2) . \]
Clearly, $\text{supp } w_2 \subseteq \widetilde{R}.$ Furthermore, using the mean value zero property of $a(x_1,x_2)$, we have:
\begin{align*}
\rizo\rizt (a)(x_1,x_2)
&=\int_{R} \bigg( \rizo(x_1,y_1)\rizt(x_2,y_2) - \rizo(x_1,x_{0,1})\rizt(x_2,x_{0,2}) \bigg) \\
&\hskip1cm\times a(y_1,y_2) \dmz(y_1)\dmz(y_2).
\end{align*}
Then following similar estimates as in $w_1$ above, we have
\begin{align*}
|w_2(x_1,x_2)|
&\ls \wz M^2  \|a\|_{L^\infty(R)}\mu_\lz(\widetilde{R}) \bigg[\frac{r_1\ m_\lz(I(x_{0,1},r_1))}
{|x_{1}-x_{0,1}| m_\lz(I(x_{0,1},|x_{1}-x_{0,1}|)) m_\lz(I(x_{0,2},|x_{2}-x_{0,2}|)) }  \\
&\quad\quad+ \frac{r_2\ m_\lz(I(x_{0,2},r_2))}
{|x_{2}-x_{0,2}| m_\lz(I(x_{0,2},|x_{2}-x_{0,2}|)) m_\lz(I(x_{0,1},|x_{1}-x_{0,1}|)) } \bigg]\chi_{\wz R}(x_1,x_2)\\
&\ls \bigg[\frac{ 1}
{  m_\lz(I(x_{0,1},|x_{1}-x_{0,1}|)) m_\lz(I(x_{0,2},|x_{2}-x_{0,2}|)) }  \\
&\quad\quad+ \frac{ 1}
{  m_\lz(I(x_{0,2},|x_{2}-x_{0,2}|)) m_\lz(I(x_{0,1},|x_{1}-x_{0,1}|))  } \bigg]\chi_{\wz R}(x_1,x_2)\\
&\ls D_2\chi_{\wz R}(x_1,x_2).
 \end{align*}

Combining the estimates of $w_1$ and $w_2$, we can conclude that $a - \Pi(f,g)$ has support contained in
$$R\cup \widetilde{R}$$ and satisfies
$$ \int_{\R_+\times\R_+}\left( a(x_1,x_2) -  \Pi(f,g) (x_1,x_2) \right) \dmz(x_1)\dmz(x_2)=0.$$
Then, from Lemma \ref{lem:H1Est}, we have
\begin{align*}
\| a - \Pi(f,g) \|_\hop& \ls \bigg( \log_2 {|x_{0,1}-y_{0,1}| \over r_1}+\log_2 {|x_{0,2}-y_{0,2}| \over r_2} \bigg)\Big(D_1\mu_\lz(R)+D_2\mu_\lz(\widetilde{R}) \Big)\\
& \ls \bigg( \log_2 {|x_{0,1}-y_{0,1}| \over r_1}+\log_2 {|x_{0,2}-y_{0,2}| \over r_2} \bigg)\Big( {r_1\over | x_{0,1}-y_{0,1} | }+{r_2\over | x_{0,2}-y_{0,2} | } \Big)\\
&\ls {\log_2\wz M\over \wz M}\\
&\ls \epsilon.
\end{align*}

\medskip

Case (b): $x_{0,1}> 2\wz Mr_1$,  $x_{0,2}\le 2\wz Mr_2$.

\smallskip
In this case, let $y_{0,1}:=x_{0,1}- {\wz Mr_1\over K_0}$ and $y_{0,2}:=x_{0,2}+2\wz MK_0r_2$ and
$$\wz R:=I(y_{0,1}, r_1)\times I(y_{0,2}, r_2).$$ We also let $g$ and $h$ be the same as in \eqref{g} and \eqref{h}, respectively.

Then ${2K_0-1\over 2K_0}x_{0,1}<y_{0,1}<x_{0,1}$.
For every $y_1\in I(y_{0,1}, r_1)$, from the facts that $K_0>\max\{{1\over K_2},\,\frac1{{K}_3}\}+1$ and $M\ge100K_0$, we have
$$0<{x_{0,1}\over y_1} -1<{K}_3.$$

To continue, for the first variable, we use Proposition \ref{p-estimate of riesz kernel} (iii)
and the fact that $y_1\sim y_{0,1}\sim x_{0,1}$ for any $y_1\in I(y_{0,1}, r_1)$; and for the second variable, we use Proposition \ref{p-estimate of riesz kernel} (ii)
the fact that $y_2/x_{0,2}>K_2^{-1}$ for any $y_2\in I(y_{0,2},r_2)$. Then we see that
\begin{align}\label{lower bound of riesz 2}
&\lf| \wrizo \wrizt (g)(x_{0,1},x_{0,2}) \r|\\
&\quad=\Big|\int_{y_{0,1}-r_1}^{y_{0,1}+r_1}  \wrizo(y_1, x_{0,1}) \dmz(y_1)\int_{y_{0,2}-r_2}^{y_{0,2}+r_2}  \wrizo(y_2, x_{0,2}) \dmz(y_2)\Big|\noz\\
&\quad\gs\int_{y_{0,1}-r_1}^{y_{0,1}+r_1}\frac1{x_{0,1}^\lz y_{0,1}^\lz} {1\over x_{0,1}-y_1}\dmz(y_1)\int_{y_{0,2}-r_2}^{y_{0,2}+r_2}\frac1{y_2}dy_2\noz\\
&\quad  \sim \int_{y_{0,1}-r_1}^{y_{0,1}+r_1} {1\over x_{0,1}-y_{0,1}}dy_1\ \frac {r_2}{y_{0,2}} \sim\frac{1}{\wz M^2}.\noz
\end{align}

Thus, from \eqref{lower bound of riesz 2},  we have that
\begin{align*}
 \|g \|_{\ltp}  \| h \|_{\ltp}
 \ls \wz M^2 \bigg({  y_{0,1}^{2\lz}\, r_1\ y_{0,2}^{2\lz}\, r_2 \over   x_{0,1}^{2\lz}\, r_1\ x_{0,2}^{2\lz}\, r_2  } \bigg)^{1\over2} \ls \wz M^{2+\lz}.
\end{align*}

Then to estimate $a(x_1,x_2) - \Pi(g,h)(x_1,x_2)$,  we define $w_1$ and $w_2$ to be the same as in Case (a). And following the same estimates as in Case (a),
we obtain that
$$ w_1(x_1,x_2)\ls  D_1\chi_{R}(x_1,x_2)\quad {\rm and}\quad w_2(x_1,x_2)\ls  D_2\chi_{\wz R}(x_1,x_2).  $$
Then, the fact that  $\| a - \Pi(f,g) \|_\hop \ls \epsilon$ now immediately follows from Lemma \ref{lem:H1Est} and the argument in Case (a).

\medskip

Case (c): $x_{0,1}\leq 2\wz Mr_1$,  $x_{0,2}> 2\wz Mr_2$.

\smallskip
In this case, let $y_{0,1}:=x_{0,1}+2\wz MK_0r_1$  and $y_{0,2}:=x_{0,2}- {\wz Mr_2\over K_0}$ and
$$\wz R:=I(y_{0,1}, r_1)\times I(y_{0,2}, r_2).$$ We also let $g$ and $h$ be the same as in \eqref{g} and \eqref{h}, respectively. Then, by handling the estimates symmetrically to Case (b), we obtain that
\begin{align}\label{lower bound of riesz 3}
&\lf| \wrizo \wrizt (g)(x_{0,1},x_{0,2}) \r| \gs\frac{1}{\wz M^2},
\end{align}
which gives
\begin{align*}
 \|g \|_{\ltp}  \| h \|_{\ltp}
 \ls \wz M^{2+\lz}.
\end{align*}
Again we obtain that   $\| a - \Pi(f,g) \|_\hop \ls \epsilon$.

\medskip

Case (d): $x_{0,1}> 2\wz Mr_1$,  $x_{0,2}> 2\wz Mr_2$.

\smallskip
In this case, let $y_{0,1}:=x_{0,1}- {\wz Mr_1\over K_0}$   and $y_{0,2}:=x_{0,2}- {\wz Mr_2\over K_0}$ and
$$\wz R:=I(y_{0,1}, r_1)\times I(y_{0,2}, r_2).$$ We also let $g$ and $h$ be the same as in \eqref{g} and \eqref{h}, respectively.
Then for $i=1,2$,  ${2K_0-1\over 2K_0}x_{0,i}<y_{0,i}<x_{0,i}$.
For every $y_i\in I(y_{0,i}, r_i)$, from the facts that $K_0>\max\{{1\over K_2},\,\frac1{{K}_3}\}+1$ and $\wz M\ge100K_0$, we have
$$0<{x_{0,i}\over y_i} -1<{K}_3.$$

To continue,  we use Proposition \ref{p-estimate of riesz kernel} (iii)
and the fact that $y_i\sim y_{0,i}\sim x_{0,i}$ for any $y_i\in I(y_{0,i}, r_i)$ for $i=1,2$. Then we see that
\begin{align}\label{lower bound of riesz 4}
&\lf| \wrizo \wrizt (g)(x_{0,1},x_{0,2}) \r|\\
&\quad\gs\int_{y_{0,1}-r_1}^{y_{0,1}+r_1}\frac1{x_{0,1}^\lz y_{0,1}^\lz} {1\over x_{0,1}-y_1}\dmz(y_1)\int_{y_{0,2}-r_2}^{y_{0,2}+r_2}\frac1{x_{0,2}^\lz y_{0,2}^\lz} {1\over x_{0,2}-y_2}\dmz(y_2)\noz\\
&\quad  \sim \int_{y_{0,1}-r_1}^{y_{0,1}+r_1} {1\over x_{0,1}-y_{0,1}}dy_1\ \int_{y_{0,2}-r_2}^{y_{0,2}+r_2} {1\over x_{0,2}-y_{0,2}}dy_2 \sim\frac{1}{\wz M^2}.\noz
\end{align}
Thus, from \eqref{lower bound of riesz 4},  we have that
\begin{align*}
 \|g \|_{\ltp}  \| h \|_{\ltp}
 \ls \wz M^2 \bigg({  y_{0,1}^{2\lz}\, r_1\ y_{0,2}^{2\lz}\, r_2 \over   x_{0,1}^{2\lz}\, r_1\ x_{0,2}^{2\lz}\, r_2  } \bigg)^{1\over2} \ls M^{2}.
\end{align*}
Again we obtain that   $\| a - \Pi(f,g) \|_\hop \ls \epsilon$.
\end{proof}

\begin{proof}[\bf Proof of Theorem  \ref{thm: weak factotrization}]
We first point out  that from \eqref{boundedness of bilinear form}, for every $g, h\in L^\infty(\rlz)$ with compact support,
$$\|\Pi(g, h)\|_\hop\ls \|g\|_\ltp\|h\|_\ltp.$$
Based on this upper bound, for every $f\in\hop$ having the representation \eqref{represent of H1} with
$$\sum_{k=1}^\fz\sum_{j=1}^\fz\lf|\az^k_j\r|\lf\|g^k_j\r\|_\ltp\lf\|h^k_j\r\|_\ltp<\fz,$$
we have that
\begin{eqnarray*}
\|f\|_\hop&\ls&\sum_{k=1}^\fz\sum_{j=1}^\fz|\alpha^k_j|\,  \Big\|  \Pi\lf(g^k_j,h^k_j\r)\Big\|_{\hop}\ls \sum_{k=1}^\fz\sum_{j=1}^\fz\lf|\az^k_j\r|\lf\|g^k_j\r\|_\ltp\lf\|h^k_j\r\|_\ltp,
\end{eqnarray*}
which gives that
\begin{eqnarray*}
\|f\|_\hop&\ls&\inf\lf\{\sum_{k=1}^\fz\sum_{j=1}^\fz\lf|\az^k_j\r|\lf\|g^k_j\r\|_\ltp\lf\|h^k_j\r\|_\ltp:
f=\sum_{k=1}^\fz\sum_{j=1}^\fz \az^k_j\,\Pi\lf(g^k_j,h^k_j\r)\r\}.
\end{eqnarray*}

It remains to show that for every $f\in\hop$, $f$ has a representation as in \eqref{represent of H1} with
\begin{equation}\label{lower bound of H1 respent}
\inf\lf\{\sum_{k=1}^\fz\sum_{j=1}^\fz\lf|\az^k_j\r|\lf\|g^k_j\r\|_\ltp\lf\|h^k_j\r\|_\ltp:
f=\sum_{k=1}^\fz\sum_{j=1}^\fz \az^k_j\,\Pi\lf(g^k_j,h^k_j\r)\r\}\ls \|f\|_\hop.
\end{equation}
To this end, assume that $f$ has the following atomic representation
 $\displaystyle f=\sum_{j=1}^\fz\az^1_ja^1_j$ with $\displaystyle\sum_{j=1}^\fz|\az^1_j|\le \wz C_0\|f\|_\hop$
 for certain absolute constant $\wz C_0\in(1, \fz)$.
We show that for every $\epsilon\in\left(0, \wz C_0^{-1}\right)$ and every $K\in\nn$, $f$ has the following representation
\begin{equation}\label{itration}
f=\sum_{k=1}^K\sum_{j=1}^\fz\az^k_j\,\Pi\lf(g^k_j, h^k_j\r)+E_K,
\end{equation}
where
\begin{equation}\label{itration-2}
\sum_{j=1}^\fz\lf|\az^k_j\r|\le  \ez^{k-1}\wz C_0^k\|f\|_\hop,
\end{equation}
and $E_K\in \hop$ with
\begin{equation}\label{itration-3}
\|E_K\|_\hop\le (\ez \wz C_0)^K\|f\|_\hop,
\end{equation}
and $g^k_j\in\ltp$, $h^k_j\in\ltp$ for each $k$ and $j$,
$\{\az^k_j\}_{j}\in \ell^1$ for each $k$ satisfying that
\begin{equation}\label{itration-1}
\lf\|g^k_j\r\|_\ltp\lf\|h^k_j\r\|_\ltp\ls C(\ez)
\end{equation}
with the absolute constant $C(\ez)= \wz M^{2+2\lz}$, where $M$ is the constant in the proof of Theorem \ref{thm:recfac} satisfying $\wz M\ge100K_0$
and $\frac{\log_2 \wz M}{\wz M}<\ez$.

In fact, for given $\ez$ and each $a^1_j$,  by Theorem \ref{thm:recfac} we obtain that
 there exist
 $g^1_j\in\ltp$ and $h^1_j\in\ltp$ with
 $$\lf\|g^1_j\r\|_\ltp\lf\|h^1_j\r\|_\ltp\ls C(\ez)$$
  and
 $$\lf\| a^1_j-\Pi\lf(g^1_j,h^1_j\r)\r\|_\hop<\ez.$$

Now we write
\begin{align*}
f&=\sum_{j=1}^\fz\az^1_ja^1_j
=\sum_{j=1}^\fz\az^1_j\Pi\lf(g^1_j,h^1_j\r)+ \sum_{j=1}^\fz\az^1_j\lf[ a^1_j-\Pi\lf(g^1_j,h^1_j\r)\r]=:M_1+E_1.
\end{align*}
Observe that
\begin{align*}
\|E_1\|_\hop&\le \sum_{j=1}^\fz\lf|\az^1_j\r| \lf\| a^1_j-\Pi\lf(g^1_j,h^1_j\r)\r\|_\hop \le \ez \wz C_0\|f\|_\hop.
\end{align*}

Since $E_1\in\hop$, for the given $\wz C_0$,
there exists a sequence of atoms $\{a^2_j\}_j$ and numbers $\{\az^2_j\}_j$
such that $\displaystyle E_1=\sum_{j=1}^\fz\az^2_ja^2_j$ and
\begin{equation*}
\sum_{j=1}^\fz\lf|\az^2_j\r|\le \wz C_0\|E_1\|_\hop\le \ez \wz C_0^2\|f\|_\hop.
\end{equation*}
Again,  we have that for given $\ez$, there exists a representation of $E_1$ such that
\begin{align*}
E_1&=\sum_{j=1}^\fz\az^2_j\Pi\lf(g^2_j,h^2_j\r)+ \sum_{j=1}^\fz\az^2_j\lf[ a^2_j-\Pi\lf(g^2_j,h^2_j\r)\r]=:M_2+E_2,
\end{align*}
and
\begin{equation*}
\lf\|g^2_j\r\|_\ltp\lf\|h^2_j\r\|_\ltp\ls C(\ez)\,\ {\rm and}\,\, \lf\| a^2_j-\Pi\lf(g^2_j, h^2_j\r)\r\|_\hop<\frac{\ez}{2}.
\end{equation*}
Moreover,
\begin{eqnarray*}
\|E_2\|_\hop&\le& \sum_{j=1}^\fz\lf|\az^2_j\r| \lf\| a^2_j-\Pi\lf(g^2_j,h^2_j\r)\r\|_\hop
\le (\ez \wz C_0)^2\|f\|_\hop.
\end{eqnarray*}
Now we conclude that
\begin{eqnarray*}
f=\sum_{j=1}^\fz\az^1_ja^1_j=\sum_{k=1}^2\sum_{j=1}^\fz\az^k_j\Pi\lf(g^k_j, h^k_j\r)+E_2,
\end{eqnarray*}

Continuing in this way, we deduce that for every $K\in\nn$, $f$ has the representation \eqref{itration} satisfying
\eqref{itration-1}, \eqref{itration-2}, and \eqref{itration-3}. Thus letting $K\to\fz$, we
see that \eqref{represent of H1} holds. Moreover, since $\ez \wz C_0<1$, we have that
$$\sum_{k=1}^\fz\sum_{j=1}^\fz \lf|\az^k_j\r|\le \sum_{k=1}^\fz\ez^{-1}(\ez \wz C_0)^k\|f\|_\hop\ls \|f\|_\hop,$$
which implies \eqref{lower bound of H1 respent} and hence, completes the proof of Theorem  \ref{thm: weak factotrization}.
\end{proof}

\bigskip
\begin{proof}[\bf Proof of (i)$\Longleftarrow$(iv)]
Suppose that $b\in L^2_{loc}(\rlz)$. Assume that $\mbhh$ is bounded on $\ltp$.

From the definition of $h^1(\rlz)$, given $f\in h^1(\rlz)$, there exists a number sequence $\{\lambda_j\}_{j=1}^\infty$
and atoms $\{a_j\}_{j=1}^\infty$ such that
$$ f = \sum_{j=1}^\infty\lambda_j a_j,$$
where the series converges in the $h^1(\rlz)$ norm and
$ \|f\|_{h^1(\rlz)}\approx \sum_{j=1}^\infty |\lambda_j|. $
Hence, we have that $f_N:= \sum_{j=1}^N \lambda_j a_j$ tends to $f$ as $N\to +\infty$ in the
 $h^1(\rlz)$ norm, which implies that $ h^1(\rlz)\cap L_c^\infty(\rlz)$ is dense in $h^1(\rlz)$, where recall that $L_c^\infty(\rlz)$ is the subspace of $L^\infty(\rlz)$ consisting of functions with compact support in $\R_+\times\R_+$.

Now for $f\in\hop \cap L_c^\infty(\rlz)$,
from Theorem \ref{thm: weak factotrization},
we choose a weak factorization of $f$ such that
\begin{eqnarray}\label{eeeee}
f&=&\sum_{k=1}^\fz\sum_{j=1}^\fz \az^k_j  \Pi\lf(g^k_j,h^k_j\r)
\end{eqnarray}
in the sense of $h^1(\rlz)$, where the sequence $\{ \az^k_j \}\in \ell^1$ and the functions
$g^k_j$ and $h^k_j$ are in $L^\infty_c(\rlz)$ satisfying
$$\sum_{k=1}^\fz\sum_{j=1}^\fz \lf|\az^k_j\r|\lf\|g^k_j\r\|_\ltp\lf\|h^k_j\r\|_\ltp \ls \|f\|_{h^1(\rlz)}.$$
From the definition of bilinear form $\Pi$ as in \eqref{bilinear form}, we see that
$\Pi\lf(g^k_j,h^k_j\r)$ is in $L^2(\rlz)$ with compact support.

Since $f\in \hop \cap L_c^\infty(\rlz)$, we see that $f$ is in $L^{2}(U)$, where we use the set $U$ to denote the support of $f$. Hence,
$$ \int_{\R_+\times\R_+} b(x_1,x_2)f(x_1,x_2)\,\dmz(x_1)\dmz(x_2) $$
is well-defined, since $b\in L^2_{loc}(\rlz)$ and hence in $L^2(U)$.

We now define
$$ b_i(x_1,x_2) = b(x_1,x_2) \chi_{\{(x_1,x_2)\in \R_+\times\R_+:\ |b(x_1,x_2)|\leq i\}}(x_1,x_2),\quad i=1,2,... $$
It is clear that $b_i(x_1,x_2)\to b(x_1,x_2)$ as $i\to\infty$ in the sense of $L^2(U)$.
And then we have
$$ \int_{\R_+\times\R_+} b(x_1,x_2)f(x_1,x_2)\,\dmz(x_1)\dmz(x_2) =\lim_{i\to\infty}  \int_{\R_+\times\R_+} b_i(x_1,x_2)f(x_1,x_2)\,\dmz(x_1)\dmz(x_2). $$
Next, for each $i=1,2,\ldots$,
we have that
\begin{align*}
&\int_{\R_+\times\R_+} b_i(x_1,x_2)f(x_1,x_2)\,\dmz(x_1)\dmz(x_2)\\
 &=\int_{\R_+\times\R_+} b_i(x_1,x_2) \sum_{k=1}^\fz\sum_{j=1}^\fz \az^k_j  \Pi\big(g^k_j,h^k_j\big)(x_1,x_2)\,\dmz(x_1)\dmz(x_2)   \\
 &=\sum_{k=1}^\fz\sum_{j=1}^\fz \az^k_j  \int_{\R_+\times\R_+} b_i(x_1,x_2)  \Pi\big(g^k_j,h^k_j\big)(x_1,x_2)\,\dmz(x_1)\dmz(x_2)   \\
  &=\sum_{k=1}^\fz\sum_{j=1}^\fz \az^k_j \langle b_i, \Pi\big(g^k_j,h^k_j\big)\rangle_{L^2(\rlz)}
\end{align*}
since $b_i$ is in $L^\infty(U)$ and hence is in ${\rm bmo}(\rlz)$,  \eqref{eeeee} holds in $h^1(\rlz)$ and
each $\Pi\big(g^k_j,h^k_j\big)$ is in $h^1(\rlz)$ as showed in Proposition \ref{t-H1 estimate of pi}.

As a consequence, we obtain that
\begin{align}\label{eeee}
|\langle b,f \rangle_{L^2(\rlz)}| &= \lim_{i\to\infty} \bigg|  \int_{\R_+\times\R_+} b_i(x_1,x_2)f(x_1,x_2)\,\dmz(x_1)\dmz(x_2) \bigg|\\
&\leq  \lim_{i\to\infty} \sum_{k=1}^\fz\sum_{j=1}^\fz |\az^k_j|\,| \langle b_i, \Pi\big(g^k_j,h^k_j\big)\rangle_{L^2(\rlz)} |\nonumber\\
&=   \sum_{k=1}^\fz\sum_{j=1}^\fz \lim_{i\to\infty} |\az^k_j|\,| \langle b_i, \Pi\big(g^k_j,h^k_j\big)\rangle_{L^2(\rlz)} |,\nonumber
\end{align}
where the equality above holds since all the terms are non-negative. Next,  since $b_i(x_1,x_2)\to b(x_1,x_2)$ as $i\to\infty$ in the sense of $L^2(V)$ and $\Pi\big(g^k_j,h^k_j\big)$ is in  $L^{2}(V)$ with $V$ the support of $\Pi\big(g^k_j,h^k_j\big)$, we have that
$$ \lim_{i\to\infty} \langle b_i,\Pi\big(g^k_j,h^k_j\big)\rangle_{L^2(\rlz)} = \langle b,\Pi\big(g^k_j,h^k_j\big)\rangle_{L^2(\rlz)}, $$
which implies that
$$ \lim_{i\to\infty}  |\langle b_i,\Pi\big(g^k_j,h^k_j\big)\rangle_{L^2(\rlz)}| = |\langle b,\Pi\big(g^k_j,h^k_j\big)\rangle_{L^2(\rlz)}|. $$
This, together with \eqref{eeee}, yields that
\begin{align*}
|\langle b,f \rangle_{L^2(\rlz)}| &\leq \sum_{k=1}^{\infty} \sum_{s=1}^\infty |\az^k_j|\,| \langle b, \Pi\big(g^k_j,h^k_j\big)\rangle_{L^2(\rlz)} |\\
&=\sum_{k=1}^\fz\sum_{j=1}^\fz |\az^k_j| \cdot\Big|\lf\langle g^k_j,\mbhh h^k_j\r\rangle_{\ltp}\Big|,
\end{align*}
which is further bounded by
\begin{eqnarray*}
&&\sum_{k=1}^\fz\sum_{j=1}^\fz |\az^k_j|\lf\|g^k_j\r\|_\ltp \big\|\mbhh h^k_j\big\|_\ltp\\
&&\le\lf\|\mbhh:\ltp\to\ltp \r\| \, \sum_{k=1}^\fz\sum_{j=1}^\fz |\az^k_j|\big\|g^k_j\big\|_\ltp\big\|h^k_j\big\|_\ltp\\
& &\ls\lf\|\mbhh:\ltp\to\ltp\r\|\|f\|_\hop.
\end{eqnarray*}
Then by the fact that $\{f\in \hop: \,f {\rm\ has\ compact\ support} \}$ is dense in $\hop$, and
the duality between $\hop$ and $\bmop$ (see Theorem \ref{thm dual}), we finish the proof.
\end{proof}


\begin{proof}[\bf Proof of Corollary \ref{coro}]
Suppose $b\in \lbmo(\rlz)$. Then based on (iii) of Theorem \ref{thm main2}, we obtain that
 there exist $f_1,f_2,g_1,g_2\in L^\infty(\rlz)$ such that $b=f_1+\rizo g_1=f_2+\rizo g_2$ and moreover,
$\|b\|_{\lbmo(\rlz)}\approx \inf\Big\{ \max_{i=1,2} \big\{\|f_i\|_{L^\infty(\rlz)},\|g_i\|_{L^\infty(\rlz)}\big\} \Big\}$ where the infimum is taken over all possible decompositions of
$b$.

We now show that $b$ is also in $\bmo_{\Delta_\lz}(\rlz)$. To see this, we recall the recent result of decomposition of
 $\bmo_{\Delta_\lz}(\rlz)$ obtained in \cite{DLWY2}.

\begin{thm}[\cite{DLWY2}]\label{thm BMO decomposition}
The following two statements are equivalent.

${\rm (i)}$ $\varphi \in {\bmo}_{\Delta_\lambda}( \rlz ) $;

${\rm(ii)}$ There exist $h_i\in L^\infty( \rlz )$, $i=1,2,3,4$, such that
$$ \varphi= h_1 +  R_{\Delta_\lambda,\,1}(h_2) + R_{\Delta_\lambda,\,2}(h_3) + R_{\Delta_\lambda,\,1}R_{\Delta_\lambda,\,2}(h_4).$$
\end{thm}

Back to the proof, we now choose $h_1=f_1$, $h_2=g_1$, $h_3=h_4=0$. Then it is easy to see that
$$ b= h_1 +  R_{\Delta_\lambda,\,1}(h_2) + R_{\Delta_\lambda,\,2}(h_3) + R_{\Delta_\lambda,\,1}R_{\Delta_\lambda,\,2}(h_4),$$
which implies that $b\in \bmo_{\Delta_\lz}(\rlz)$.

Similarly, we can also choose $h_1=f_2$, $h_3=g_2$, $h_2=h_4=0$. Combining these two choices, we further obtain that
$$ \|b\|_{\bmo_{\Delta_\lz}(\rlz)}\ls \|b\|_{\lbmo(\rlz)},$$
which implies that
$$ \lbmo(\rlz)\subset \bmo_{\Delta_\lz}(\rlz).$$

Next we prove that $\lbmo_{\Delta_\lz}(\rlz)$ is a proper subspace of $\bmo_{\Delta_\lz}(\rlz)$. To see this,
we let $K_3$ be the constant in (iii) of Proposition \ref{p-estimate of riesz kernel}.
Since $\rizo\rizt$ is a product Calder\'on--Zygmund operator on $\rlz$ and hence it is bounded from $L^\infty(\rlz)$ to $\bmo_{\Delta_\lz}(\rlz)$ (see \cite{HLL}). Then, it is direct that the following function
\begin{align}\label{a BMO function}
b(x_1,x_2):=\rizo\rizt(\chi_{(1,2)\times(1,2)})(x_1,x_2)
\end{align}
is in $\bmo_{\Delta_\lz}(\rlz)$.

Next we claim that this function $b(x_1,x_2)$ is not in $\lbmo(\rlz)$. To see this,
we first note that $b(x_1,x_2)$ can be written as
\begin{align*}\label{a BMO function}
b(x_1,x_2)=\riz(\chi_{(1,2)})(x_1)\riz(\chi_{(1,2)})(x_2).
\end{align*}
We now verify that $\riz(\chi_{(1,2)})(x_1)$
is not in $L^\infty(\R_+,\dmz)$.
%
%
In fact, by Proposition \ref{p-estimate of riesz kernel}, for every $\delta>0$ small enough and $x_1\in (1-\delta,1)$, we choose $\epsilon=2\delta$. Then we have
\begin{align*}
\riz(\chi_{(1,2)})(x_1)&=\int_1^{2} \riz(x_1,y)y^{2\lz}dy
\geq \int_{x_1+\epsilon}^{(1+K_3)x_1} \riz(x_1,y)y^{2\lz}dy\\
&\geq  \int_{x_1+\epsilon}^{(1+K_3)x_1}  C_{K_3,\lz}{1\over x_1^\lambda y^\lambda}{1\over y-x_1}y^{2\lz}dy\\
&\gs\int_{x_1+\epsilon}^{(1+K_3)x_1}  {1\over y-x_1}dy\\
&= \ln(y-x_1)\Big|_{x_1+\epsilon}^{(1+K_3)x_1}\\
&=\ln(K_3x_1)-\ln\epsilon\\
&=\ln(K_3x_1)-\ln(2\delta).
\end{align*}
Then it is direct that when $\delta\to 0^+$, $\riz(\chi_{(0,1)})(x_1)$ is unbounded around the interval $(1-\delta,1)$.

Hence, for the function $b(x_1,x_2)$ defined as in \eqref{a BMO function}, when we fix $x_1$, $b(x_1, x_2)$ as a function of $x_2$ is in $\bmo_{\Delta_\lz}(\R_+,\dmz)$. However, it is not uniform for the variable $x_1$.
\end{proof}

{\bf Acknowledgments:}
X. T. Duong and J. Li are supported by ARC DP 160100153. J. Li is also supported by Macquarie University New Staff Grant.
B. D. Wick's research supported in part by National Science Foundation
DMS grants \#1603246 and \#1560955.
D. Yang is supported by the NNSF of China (Grant No. 11571289). The authors would like to thank the referee for careful reading of the paper and for
helpful comments and suggestions.


\begin{thebibliography}{DLMWY}






%
%
%
%
%
%
%



\bibitem[AK]{ak} K. F. Andersen and R. A. Kerman, Weighted norm inequalities for
generalized Hankel conjugate transformations,  \emph{Studia Math.},
{ \bf 71} (1981/82), 15--26.


\bibitem[AH]{AuHyt} %
    {P. Auscher and T. Hyt\"onen}, %
    Orthonormal bases of regular wavelets in spaces of homogeneous type, %
    \emph{Appl. Comput. Harmon. Anal.}~\textbf{34} (2013), 266--296.


\bibitem[BCFR]{bcfr} J. J. Betancor, A. Chicco Ruiz,
J. C. Fari\~{n}a and L. Rodr\'iguez-Mesa,
Maximal operators, Riesz transforms and Littlewood-Paley functions
associated with Bessel operators on BMO,
\emph{J. Math. Anal. Appl.}, {\bf  363}  (2010), 310-326.


\bibitem[BDT]{bdt} J. J. Betancor, J. Dziuba\'nski and J. L. Torrea,
On Hardy spaces associated with Bessel operators, {\it J. Anal. Math.} {\bf107}
(2009) 195-219.



\bibitem[BFBMT]{bfbmt} J. J. Betancor, J. C. Fari\~na, D. Buraczewski,
T. Mart\'\i nez and J. L. Torrea, Riesz transform related to Bessel operators,
{\it Proc. Roy. Soc. Edinburgh Sect. A}, {\bf 137} (2007), 701-725.


\bibitem[BFS]{bfs} J. J. Betancor, J. C. Fari\~{n}a and A. Sanabria,
On Littlewood-Paley functions associated with Bessel operators,
{\it  Glasg. Math. J.}, {\bf 51} (2009), 55--70.


\bibitem[BHNV]{bhnv} J. J. Betancor, E. Harboure, A. Nowak and B. Viviani,
Mapping properties of functional operators in harmonic analysis related
to Bessel operators, {\it  Studia Math.}, {\bf 197} (2010), 101--140.


%
%

\bibitem[BC]{bc} M. Bramanti and M. C. Cerutti, Commutators of singular integrals on homogeneous spaces,
\emph{ Boll. Un. Mat. Ital. B (7)} \textbf{10} (1996), 843--883.




\bibitem[CF]{CF} %
    {S.-Y.A.~Chang and R.~Fefferman}, %
    A continuous version of duality of $H^1$ with $\bmo$ on the
    bidisc, %
    \emph{Annals of Math.}~\textbf{112} (1980), 179--201.






%


\bibitem[CRW]{crw} R. R. Coifman, R. Rochberg and G. Weiss,
Factorization theorems for Hardy spaces in several variables, \emph{Ann. of Math. (2)} 103 (1976), 611-635.




\bibitem[CW77]{CW2}
    {R. R.~Coifman and G.~Weiss}, %
    Extensions of Hardy spaces and their use in analysis, %
    \emph{Bull.\ Amer.\ Math.\ Soc.}~\textbf{83} (1977), 569--645.





%






\bibitem[DO]{DO} L. Dalenc and Y. Ou,
Upper bound for multi-parameter iterated commutators, \emph{Publ. Mat.} \textbf{60} (2016), 191--220.

\bibitem[DLMWY]{DLMWY} X. T. Duong, J. Li, S. Mao, H. Wu and D. Yang, Compactness of Riesz transform commutator associated with Bessel operators, {\it J. Anal. Math.}, to appear.


\bibitem[DLWY]{DLWY} X. T. Duong, J. Li, B. D. Wick and D. Yang, Hardy space via factorization, and BMO
space via commutators in the Bessel setting, \emph{Indiana Univ. Math. J.}, to appear.


\bibitem[DLWY2]{DLWY2} X. T. Duong, J. Li, B. D. Wick and D. Yang,
Characterizations of product Hardy spaces in Bessel setting,
 arXiv:1606.03500.

\bibitem[DLWY3]{DLWY3} X. T. Duong, J. Li, B. D. Wick and D. Yang,
 Commutators, little bmo and weak factorization, \emph{Annales de l'Institut Fourier.}, to appear.







%





\bibitem[FS]{fs}
S.~H. Ferguson and C. Sadosky,
Characterizations of bounded mean oscillation on the polydisk in terms of Hankel operators and Carleson measures,
\emph{J. Anal. Math.} \textbf{81} (2000), 239--267.



\bibitem[FL]{fl}
S.~H. Ferguson and M.~T. Lacey,
\newblock A characterization of product {BMO} by commutators,
\newblock {\em Acta Math.}, {\bf189} (2002), 143--160.


%














\bibitem[Gra]{Gra}
L. Grafakos,
Classical Fourier Analysis, Second Edition, Graduate Texts in Mathematics, {\bf249},
Springer,
2008.






%
\bibitem[H]{h}D. T. Haimo, Integral equations associated
with Hankel convolutions, \emph{Trans. Amer. Math. Soc.}, {\bf116} (1965), 330--375.


\bibitem[HLL]{HLL} %
    {Y.~Han, J.~Li, and C.-C.~Lin}, %
    {Criterions of the $L^2$ boundedness and sharp endpoint estimates for singular integral operators on product spaces of homogeneous type}, accepted by \emph{Ann. Scuola Norm. Sup. Pisa Cl. Sci.}, 2015.


%
\bibitem[HLL]{HLL2} Y. Han, J. Li and G. Lu, {
    Multiparameter Hardy space theory on Carnot-Carath\'eodory
    spaces and product spaces of homogeneous type}, \emph{Trans.
    Amer. Math. Soc.} {\bf 365} (2013), 319--360.
%
%
%
\bibitem[HLW]{HLW} %
    {Y.~Han, J.~Li, and L.A.~Ward}, %
    Hardy space theory on spaces of homogeneous type via orthonormal
    wavelet bases, %
    \emph{Applied and Computational Harmonic Analysis}, to appear. %
%
%
%
\bibitem[HMY1]{hmy06} Y. Han, D. M\"uller and D. Yang,
Littlewood-Paley characterizations for Hardy spaces on spaces of
homogeneous type, {\it Math. Nachr.} {\bf279} (2006), 1505-1537.
%
%
\bibitem[HMY2]{hmy} Y. Han, D. M\"uller and D. Yang,
A theory of Besov and Triebel-Lizorkin spaces
on metric measure spaces modeled on Carnot-Carath\'eodory spaces,
{\it Abstr. Appl. Anal.} {\bf2008}, Article ID 893409, 250 pp.
%





\bibitem[Hu]{Hu} A. Huber, On the uniqueness of generalized axially symmetric potentials, \emph{ Ann. Math.}, {\bf60}, (1954), 351--358.

\bibitem[Hy]{Hy} T. ~Hyt\"onen, The sharp weighted bound for general Calder\'on-Zygmund operators,
 \emph{Ann. of Math. (2)}~\textbf{175}  (2012), 1473--1506.


\bibitem[HK]{HK} %
    {T.~Hyt\"onen and A.~Kairema}, %
    Systems of dyadic cubes in a doubling metric space, %
    \emph{Colloq. Math.}~\textbf{126} (2012),  1--33.






 \bibitem[J]{J} J. L. Journ\'e,  Calder\'on-Zygmund Operators, Pseudodifferential Operators And the Cauchy Integral of Calder\'on. Lecture Notes in Mathematics, 994. Springer--Verlag, Berlin, 1983.



\bibitem[KLPW]{KLPW}
    {A. Kairema, J. Li, C. Pereyra and L.A. Ward},  Haar bases on quasi-metric measure spaces, and dyadic structure theorems for function spaces on product spaces of homogeneous type, \emph{J. Funct. Anal.}, {\bf 271} (2016), 1793--1843.


\bibitem[K]{k78} R. A. Kerman, Boundedness criteria for generalized
Hankel conjugate transformations,  \emph{Canad. J. Math.}, {\bf 30} (1978), 147--153.


\bibitem[KL]{kl} S. G. Krantz and S.-Y. Li, Boundedness and compactness of integral operators on spaces of homogeneous type and applications.
I, \emph{J. Math. Anal. Appl.} \textbf{258} (2001),  629--641.


%
%




\bibitem[LPPW]{LPPW} M. Lacey, S. Petermichl, J. Pipher and B. D. Wick,  Multiparameter Riesz commutators,
\emph{Amer. J. Math.}, {\bf131} (2009), 731--769.


\bibitem[LPPW2]{LPPW2} M. Lacey, S. Petermichl, J. Pipher and B. D. Wick, Iterated Riesz commutators: a simple proof of boundedness, Harmonic analysis
 and partial differential equations, {\em Contemp. Math.}, {\bf 505} (2010), 171--178.


%
%
%






\bibitem[MSt]{ms} B. Muckenhoupt and  E. M. Stein,
Classical expansions and their relation to conjugate harmonic functions,
\emph{Trans. Amer. Math. Soc.}, {\bf 118} (1965), 17--92.


\bibitem[NRV]{NRV} %
    {F.~Nazarov, A.~Reznikov, and A.~Volberg}, %
    The proof of $A_2$ conjecture in a geometrically doubling metric space, %
    \emph{Indiana Univ.\ Math.\ J.}~\textbf{62} (2013), 1503--1533.

\bibitem[NV]{NV}
    {F.~Nazarov, A.~Reznikov, and A.~Volberg},
    A simple sharp weighted estimate of the dyadic shifts on metric spaces with geometric doubling,
    \emph{Int. Math. Res. Not. IMRN} (2013), 3771--3789.




%






%
%






%



\bibitem[Tao]{Tao} %
    {T.~Tao}, %
    Dyadic product~$H^1$, $\bmo$, and Carleson's
    counterexample, %
    Short stories, \url{http://www.math.ucla.edu/~tao/preprints/harmonic.html}.






\bibitem[Uch]{U} A. Uchiyama,
The factorization of $H^p$ on the space of homogeneous type,
\emph{Pacific J. Math.} \textbf{92} (1981), 453-468.



\bibitem[V]{v08} M. Villani, Riesz transforms associated to Bessel operators,
\emph{Illinois J. Math.}, {\bf52} (2008), 77--89.




%
%
\bibitem[W]{w} A. Weinstein, Discontinuous integrals and generalized potential theory,
{\it Trans. Amer. Math. Soc.} {\bf63} (1948) 342--354.

\bibitem[YY]{yy}  Da. Yang and Do. Yang,
Real-variable characterizations of Hardy spaces associated with Bessel operators,
\emph{Anal. Appl. (Singap.)}, {\bf9} (2011), 345--368.


\end{thebibliography}
\end{document}